\documentclass{amsart}
\usepackage{amsmath}
\usepackage{amssymb}
\usepackage[all]{xy}
\usepackage{comment}
\usepackage{color}

\theoremstyle{plain}
\newtheorem{thm}{Theorem}[section]
\newtheorem{thm*}{Theorem}[section]
\newtheorem{cor}[thm]{Corollary}
\newtheorem{prop}[thm]{Proposition}
\newtheorem{lemma}[thm]{Lemma}

\newtheorem{lemma*}{Lemma}

\newtheorem{construct}[thm]{Construction}

\theoremstyle{definition}
\newtheorem{defn}[thm]{Definition}
\newtheorem{remark}[thm]{Remark}

\newtheorem*{remark*}{Remark}

\newtheorem{ex}[thm]{Example}
\newtheorem{notation}[thm]{Notation}

\newtheorem{question*}{Question}

\numberwithin{equation}{thm}

\newcommand{\bR}{\mathbb R}

\newcommand{\bT}{\mathbb T}
\newcommand{\bC}{\mathbb C}

\newcommand{\cB}{\mathcal B}

\newcommand{\cN}{\mathcal N}

\def\Spec{\operatorname{Spec}\nolimits}

\def\Proj{\operatorname{Proj}\nolimits}

\newcommand{\bG}{\mathbb G}

\newcommand{\cO}{\mathcal O}
\newcommand{\bA}{\mathbb A}

\newcommand{\cF}{\mathcal F}
\newcommand{\bP}{\mathbb P}
\newcommand{\bZ}{\mathbb Z}
\newcommand{\bF}{\mathbb F}

\newcommand{\cC}{\mathcal C}

\newcommand{\cE}{\mathcal E}

\newcommand{\cJ}{\mathcal J}

\newcommand{\ul}{\underline}

\def\Spec{\operatorname{Spec}\nolimits}
\def\sl2{\operatorname{SL_{2(2)}}\nolimits}
\def\Ga2{\operatorname{\mathbb G_{\rm a(2)}}\nolimits}

\newcommand{\bN}{\mathbb N}
\newcommand{\bQ}{\mathbb Q}

\newcommand{\bU}{\mathbb U}

\newcommand{\bu}{\bullet}

\date\today

\begin{document}

 \title[Reformulation of the Stable Adams Conjecture]{Reformulation of the Stable Adams Conjecture}
 
 \author[ Eric M. Friedlander]
{Eric M. Friedlander$^{*}$} 

\address {Department of Mathematics, University of Southern California,
Los Angeles, CA 90089}
\email{ericmf@usc.edu}

\thanks{$^{*}$ partially supported by the Simons Foundation }

\subjclass[2020]{55N20, 55R15}

\keywords{stable Adams conjecture, spectra, simplicial schemes}

\begin{abstract}
We revisit methods of proof of the Adams Conjecture in order to correct and 
supplement earlier efforts to prove  analogous conjectures in the stable 
homotopy category.   We utilize simplicial
schemes over an algebraically closed field of positive characteristic and
a rigid version of Artin-Mazur \'etale homotopy theory.   Consideration
of special $\cF$-spaces together with Bousfield-Kan $\bZ/\ell$-completion
 enables us to employ an
``\'etale functor" which commutes up to homotopy with products of simplicial 
schemes.   In order to prove the Stable Adams Conjecture, we construct the universal
$\bZ/\ell$-completed $X$-fibrations for various pointed
simplicial sets $X$.  Thus, two maps from a given $\cF$-space $\ul\cB$ 
to the base $\cF$-space of the universal $\bZ/\ell$-completed $X$-fibration
$\pi_{X,\ell}: \ul\cB (G_\ell(X),X_\ell) \to \ul\cB G_\ell(X)$ determine homotopy equivalent 
maps of spectra if and only they correspond via pull-back of $\pi_{X,\ell}$ to
fiber homotopy equivalent $\bZ/\ell$-completed $X$-fibrations over $\ul\cB$.  
For the proof of the Stable Adams Conjecture, we consider maps of $\cF$-spaces
$\ul\cB \to \ul\cB G_\ell(S^2)$ where $\ul\cB$ is an $\cF$-space model
of connective $\ell$-completed connective $K$-theory.
\end{abstract}

\maketitle

\tableofcontents 

\section{Introduction}

The original Adams conjecture \cite{Adams} and its refinement, the Stable Adams Conjecture, involve the
$J$-homomorphism from (real or complex) $K$-theory to the stable homotopy groups of 
spheres and the effect of composing the $J$-homomorphism with the Adams operation
$\psi^p$ on $K$-theory for some prime $p$.   The conclusion of the original Adams
Conjecture and the more refined Stable Adams Conjecture (proved below)
requires (Bousfield-Kan) $\bZ/\ell$-completion or $Z_{(\ell)}$-localization
at a prime $\ell \not= p$.

As pointed out in \cite{B-K}, the author's verification of the Stable Adams Conjecture
given in \cite{F80} is contradicted by a counter-example.  The basic error of our
earlier work is the failure to take into account that  
the Frobenius map on ``algebraic spheres" is not orientation preserving.   Dropping
the condition of orientation preserving leads to the somewhat modified statement
of Theorem \ref{thm:below} and a proof that works with 
$\bZ/\ell$-completed $S^2$-fibrations (already considered in \cite{F80}) rather than
the $\bZ/\ell$-completion of oriented $S^2$-fibrations. 

\vskip .1in

Our main theorem is the following formulation of the Stable Adams Conjecture,
occurring as Theorem \ref{thm:stable-adams} in Section \ref{sec:stable}.   The
homotopy groups of $\pi_{i+1}({\bf BS^2_\ell})$ for $i > 0$ are naturally identified
with the $\bZ/\ell$-completions of stable homotopy groups of spheres, 
$\pi_i^{stable}(S^0) \otimes \bZ_\ell$, as made explicit in Proposition \ref{prop:stable-homotopy}.

\begin{thm}
\label{thm:below}
Let ${\bf kU}$ denote the 0-connected spectrum of (topological) complex K-theory
and let ${\bf BS^{2}_\ell}$ denote a 0-connected spectrum naturally constructed using
self-equivalences of $\bZ/\ell$-completions of even spheres.

If $p$ and  $\ell$ are distinct primes, then the maps of spectra
$${\bf J}_\ell, \ {\bf J}_\ell \circ (\psi^p)_\ell:  ({\bf kU})_\ell \times_{{\bf K(\bZ_\ell,0)}} {\bf K(\bZ,0)} 
\quad \to \quad {\bf BS^2_\ell} $$
are homotopy equivalent.
\end{thm} 

\vskip .1in

The above theorem has the following oriented version, given below as 
Corollary \ref{cor:oriented-stable-adams}.
We denote by $(\bf{kU})^o$ the 1-connected cover (which is 2-connected) of $\bf{kU}$.
We denote by  ${\bf bS^2_\ell}$  the 2-connected cover of ${\bf BS^2_\ell}$
constructed using oriented self-equivalences of $\bZ/\ell$-completions of even spheres.

\vskip .1in

\begin{cor}
Let $(\bf J_\ell)^o:  (({\bf kU})^o)_\ell \ \to \ \bf{bS^2_\ell}$ denote the map induced by ${\bf J_\ell}$ 
 on 2-connnected covers.  The following maps 
  $$(\bf J_\ell)^o, \ (\bf J_\ell)^o \circ (\psi^p)_\ell:  (({\bf kU})^o)_\ell 
\quad \to \quad {\bf bS^2_\ell}  $$
are homotopy equivalent as maps of spectra whenever $p$ and $\ell$ are distinct primes.
 \end{cor}
 
 \vskip .1in
 
 Theorem \ref{thm:below} has its own history.  Inspired by the work of D. Quillen \cite{Quillen}
 and D. Sullivan \cite{Sul} and encouraged by J. F. Adams, the author published an announcement
 in 1977 \cite{F-Sey} followed by his article \cite{F80}.  As pointed out in \cite{B-K}, that article
 failed to address the fact that the Frobenius map on ``algebraic spheres" is not orientation 
 preserving.  This present work addresses the ``tension" between the need for $\bZ/\ell$-completion
 (necessary when passing to algebraic geometry in characteristic $p \not= \ell$) and the awkward
 behavior of $\bZ/\ell$-completion when dealing with spaces which are not simply connected.
 Although this paper uses concepts and techniques of the 1970's rather than more sophisticated
 concepts of stable homotopy theory (as in \cite{Ando1}, \cite{Ando2}), the concrete models
 of $\cF$-spaces (also known as Segal $\Gamma$-spaces, as introduced in  \cite{Segal} )
 can serve as a stepping stone to infinity categories and the much more sophisticated homotopy 
 theory of J. Lurie \cite{Lurie} than the homotopy theory introduced by A.K. Bousfield and the author in \cite{Bo-F}. 
 
 Our model for the homotopy category of spectra is the homotopy category of $\cF$-spaces.
 The category $\cF$ is the opposite category of Segal's category $\Gamma$; an 
 $\cF$-space is a functor ${\ul \cB}: \cF \ \to \ (s.sets_*)$ from 
the category $\cF$ of finite pointed sets to the category $(s.sets_*)$ of pointed simplicial sets. 
 As we discuss
 in this text, the $\cF$-space $ {\ul \cB}GL(\bC)$ is a model for the 0-connected spectrum of complex K-theory,
 the $\cF$-space $\ul\cB G(S^2)$ is a model for ${\bf BS^2}$, the $\cF$-space $\ul\cB G_\ell(S^2)$ is 
 a model for ${\bf BS^2_\ell}$, and $\cJ:  {\ul \cB}GL(\bC) \to {\ul \cB} G(S^2)$ \
 is a model for the $J$-homomorphism.  
 
\vskip .05in

A fundamental ingredient in the proof of the Theorem \ref{thm:below} is the following representability 
statement occurring as Theorem \ref{thm:X-ell-universal} in Section \ref{sec:ell-complete}.
In particular, this enables us to conclude when two different maps from an $\cF$-space $\ul \cB$
to the $\cF$-space $\ul\cB G_\ell(S^2)$ determine homotopy equivalent maps of spectra.
Our definition of a $\bZ/\ell$-complete $X$-fibration of $\cF$-spaces over $\ul\cN$ is a necessary elaboration
of our original definition in \cite{F80}, a modification due of P. Bhattacharya and Kitchloo
\cite{B-K}. 

\begin{thm}
\label{thm:univ}
Let $X$ be a pointed, connected simplicial set which is smash $\ell$-good.
Denote by 
$G_\ell(X) \hookrightarrow {\ul Hom}(X,X_\ell)$ the simplicial submonoid of mod-$\ell$ equivalences.
The $\bZ/\ell$-completed $X$-fibration of $\cF$-spaces over $\ul\cN$ given
in (\ref{eqn:pi-G(X)-ell}),
\ $\pi_{X,\ell}: \ul \cB (G_\ell(X),X_\ell) \ \to \ \ul \cB G_\ell(X),$ \
satisfies the following universal property:
\vskip .05in
For any special $\cF$-space ${\ul \cB}$ over $\ul\cN$ and any $\ell$-completed $X$-fibration 
$f: {\ul \cE} \ \to {\ul\cB}$ over ${\ul\cB}$,
there is a unique homotopy class of maps of $\cF$-spaces $\phi: {\ul \cB} \ \to \ {\ul\cB}G_\ell(X)$
such that $\phi^*(\pi_{X,\ell})$ is fiber homotopy equivalent to $f$ as $\bZ/\ell$-completed 
$X$-fibrations over ${\ul\cB}$.
\end{thm}

An analogous theorem, Theorem \ref{thm:X-ell-universal-o}, is proved relating oriented homotopy
classes of oriented $\bZ/\ell$-completed $X$-fibrations to homotopy classes of maps to 
${\ul\cB}G^o_\ell(X)$.  If $X$ is a pointed, connected, finite simplicial set such that $|X|$ is a
nilpotent space, then Proposition \ref{prop:pi-o} verifies that
 ${\ul\cB}G^o_\ell(X)$ is homotopy equivalent
to the $\bZ/\ell$-completion of $\ul\cB G^o(Y)$.  Our example of interest the case that $|X| = S^2$.

Formulating  the Adams operation $\psi^p$ in the context of $\cF$-spaces is not
straight-forward.  D. Quillen's approach was to represent  $\psi^p$ ``at the unstable level"
by the Frobenius map on Grassmannianians over a field of characteristic 
$p > 0$.  Although this interpretation of $\psi^p$ informs our intuition and is central
to our proof, our $\cF$-space representation also employs
D. Sullivan's ``complex Frobenius map" $\sigma^{1/p}$ which induces a discontinuous Galois
action on complex Grassmannians (see \cite{Sul}).

 Our strategy of proof is to use commutativity of diagrams of simplicial schemes
 to produce $\cF$-spaces.  To do this, we utilize the functor $(-)^\wedge$ 
which sends a simplicial scheme to a simplicial set.
Key attributes of $X \mapsto X^\wedge \equiv (X)^\wedge$ are its functoriality, its 
dependence upon the \'etale
site of $X$, its use of $\bZ/\ell$-completion for some prime $\ell$ invertible in the structure
sheaf of $X$, and the fact that the functorial map of simplicial sets $(X\times Y)^\wedge \to 
X^\wedge \times Y^\wedge$ is a homotopy equivalence for particularly nice simplicial 
schemes $X$ and $Y$.   For all but
the most intrepid reader, $(-)^\wedge$ should be viewed as a black box explained in detail in
the references.  Granted that the functorial maps 
$(X \times Y)^\wedge \ \to \ X^\wedge \times Y^\wedge$ are homotopy equivalences 
for ``relevant" simplicial $k$-varieties $X$ and $Y$, the $\cF$-spaces obtained from
``special $\cF$-objects" of simplicial schemes are  special $\cF$-spaces. 

One objective of this current work is to enable others
to use similar techniques to explore the interface of algebraic geometry and
algebraic topology.  With this in mind, we provide references for
various results in \'etale homotopy theory and for $\bZ/\ell$-completions
of ``spaces" that we use.  In particular, we rely heavily
on the fundamental text \cite{Bo-Kan} by A.K. Bousfield and D. Kan 
for important properties of the $\bZ/\ell$-completion functor
$(\bZ/\ell)_\infty(-): (s.sets_*) \to (Kan \ cxes_*)$ as well as the homotopy limit functor
$\underset{\longleftarrow}{holim}(-)$.  

The interested reader will find that we dedicate much of this text to giving explicit 
constructions of objects and maps in the category of $\cF$-spaces in order to ensure 
that our conclusions are valid in stable homotopy theory.     For the classification 
theorems, we confront the issues that smash products of Kan complexes are not 
Kan complexes and that smash product of $\bZ/\ell$-completions are not $\bZ/\ell$-complete.
This entails some care in defining function complexes and restrictions on homotopy types
whose completions we consider.  We introduce a formulation of 
 smash products of algebraic spheres which behaves well with respect to actions
 we consider.

Throughout this paper, $k$ will denote an algebraically closed field of characteristic $p > 0$ and
$\ell$ will denote a prime different from $p$.  We caution the reader that we use 
$\bZ_\ell$ to denote the $\ell$-adic integers (equal to \ $\varprojlim_n \bZ/\ell^n$).
We fix some $\ell$-th primitive root of unity in $k$, 
thereby identifying the \'etale sheaf $\mu_\ell$ on a $k$-variety with the constant sheaf $\bZ/\ell$. 
We frequently utilize \ $(X)_\ell$ or $X_\ell$ to denote $(\bZ/\ell)_\infty(X)$, 
which we refer to as the $\bZ/\ell$-completion of $X$.   On the other hand, 
we refer to the ``$\ell$-completion" of an abelian group.  
When working with schemes over $\Spec R$ for $R$ equal to
$k$, \ $\bC$, \ or the Witt vectors $W(k)$ of $k$, we designate the fiber product over $\Spec R$ 
simply by $(-) \times (-)$.

We respectfully acknowledge the influence upon this work of the innovative ideas introduced
by Daniel Quillen in \cite{Quillen} and Dennis Sullivan in \cite{Sul} to prove the original 
Adams Conjecture.  With much gratitude, we thank J. Peter May for his patient guidance 
over many years.  We also thank an anonymous mathematician who discovered 
that the ``proof" given in our paper \cite{F80} is flawed.  Finally, we thank Nitya Kitchloo who 
revived our interest in the Stable Adams Conjecture and whose modified formulation 
of the Stable Adams Conjecture influenced this revision.

\vskip .2in


\section{Frobenius maps and orientations}

We begin by recalling the (geometric) Frobenius endomorphism,
stated for simplicity for affine $k$-varieties over $k$ defined over $\bF_q$
but easily extended to simplicial varieties over $k$ defined over $\bF_q$.

\begin{prop}
\label{prop:Frob}
Let $V$ be an affine algebraic variety over $k$, the (algebraic) spectrum of
a finitely generated commutative $k$-algebra, $\cO(V)$.   Assume that $V$ is defined 
over some finite subfield $\bF_q \subset k$ with $q = p^d$, so that $\cO(V) = \cO(V_{\bF_q}) \otimes_ k$
for some finitely generated commutative $\bF_q$-algebra $\cO(V_{\bF_q})$.    
\begin{enumerate}
\item
The $d$-th {\bf (geometric) Frobenius} $F^d: V \to V$ is map of $k$-varieties defined as
the base change to $k$ of the $q$-th power map
$V_{\bF_q} \to V_{\bF_q}$ given by sending $f \in \cO(V_{\bF_q})$ to $f^q$.  
\item 
The  $d$-th {\bf (arithmetic) Frobenius} $\sigma^d: V \to V$ is the map of schemes (but
not of $k$-varieties) given
by $1 \otimes (-)^q: \cO(V_{\bF_q} ) \otimes k \ \to \ \cO(V_{\bF_q} ) \otimes k$.
\item
The {\bf total $q$-th power map} \ $(-)^q = F^d \circ \sigma^d = \sigma^d \circ F^d: V \ \to \ V$ 
is the map of schemes sending each  $f \in \cO(V)$ to $f^q$.  This map  induces the identity map 
on the (rigid) \'etale homotopy type and \`etale cohomology of $V$.
 \end{enumerate}
 \end{prop}
 
 \vskip .1in

The proof of the following proposition is a straight-forward application of the Kummer sequence 
$0 \to \cO^* \to \cO^* \to \mu_\ell \to 0$ and Hilbert's Theorem 90.  (Recall that we have fixed
an isomorphism of $\mu_\ell$ with $\bZ/\ell$.)

\begin{prop}
\label{prop:Pn}
The map on \`etale cohomology
$$F^*: H_{et}^{2n}(\bP^n_k,\mu_\ell)  \ \simeq \bZ/\ell \quad \to \quad \bZ/\ell \ \simeq \ H_{et}^{2n}(\bP^n_k,\mu/\ell)$$
induced by the Frobenius map equals multiplication by $p^n$ on $\bZ/\ell$.  Here, $\bP^n_k$
is the projective algebraic variety given by $\Proj(k[x_0,\ldots,x_n])$.
 \end{prop}

The following corollary of Proposition \ref{prop:Pn} follows from standard properties of the \'etale 
cohomology of a simplical scheme.  We point out that the definition of $S^{2n,alg}_R$ differs from
that of \cite[Prop 8.2]{F80}, enabling a corrected formulation of the smash product of ``algebraic spheres".
\vskip .1in

\begin{cor}
\label{cor:Frob}
Let $R$ denote either $\bC$ (the complex numbers) or $k$ or the the Witt vectors $W(k)$ of $k$ 
(a complete discrete valuation ring with residue field $k$ and field of fractions of characteristic 0).
Consider the simplicial mapping cone of the open embedding 
$\bA_R^{n}{\text -}\{ 0 \} \hookrightarrow \bA^n_R$,
$$S^{2n,alg}_R \ \equiv \ (\Spec R) \cup_{(\bA_R^n{\text-}\{ 0 \} \times 0)}  
(\bA_R^n{\text-}\{ 0 \} \times \Delta[1])  \cup_{(\bA_R^n{\text-}\{ 0 \} \times 1)}   \bA_R^n,$$
a simplicial $R$-scheme.  Then 
$$H_{et}^i(S^{2n,alg}_R,\bZ/\ell) \simeq \bZ/\ell, \ i = 0, 2n, \quad H_{et}^i(S^{2n,alg}_R,\bZ/\ell) = 0, \ 
i \not=  0, 2n.$$  

The Frobenius map $F$ induces $F^*: H_{et}^{2n}(S^{2n,alg}_k,\bZ/\ell)  \ \to \ 
H_{et}^{2n}(S^{2n,alg}_k,\bZ/\ell)$ which corresponds to multiplication by $p^n$ on $\bZ/\ell$.
\end{cor}

\begin{proof}
A standard simplicial argument implies that it suffices to verify that \\
$H_{et}^i(\bA_R^n{\text-}\{ 0 \} ,\bZ/\ell) \simeq \bZ/\ell, \ i = 0, 2n-1$
and $H_{et}^i(\bA_R^n{\text-}\{ 0 \} ,\bZ/\ell) = 0$ otherwise.  This follows by
induction on $n$ using Proposition \ref{prop:Pn}: one starts with $H^*(\bG_{m,R},\bZ/\ell)$
and uses the spectral sequence 
for the bundle map $\bA_R^n{\text-}\{ 0 \}  \ \to \ \bP_{n,R}^{n-1}$ with fibers isomorphic to $\bG_{m,R}$.

The same spectral sequence, now with ground field $k$, enables the identification of $F^*$ on
$H_{et}^{2n-1}(\bA_k^n{\text-}\{ 0 \} ,\bZ/\ell)$ with $F^*$ on $H_{et}^{2n}(\bP^n_k,\bZ/\ell)$
which leads to the identification of $F^*$ on $H_{et}^{2n}(S^{2n,alg}_k,\bZ/\ell)$.
\end{proof}

\vskip .1in

Smash products play an important role in our arguments.  The following definition formalizes our
the smash products of ``algebraic spheres", correcting the definition in \cite[\S 8]{F80}.

\vskip .1in

\begin{defn}
\label{defn:smash-alg}
Let $R$ denote $\bC$ or $W(k)$ or $k$.  The smash product 
\begin{equation}
\label{eqn:define-wedge}
\wedge:  S^{2m,alg}_R \times S^{2n,alg}_R \quad \to \quad S^{2m+2n,alg}_R
\end{equation}
is the map of simplicial $R$-schemes induced by the natural map
from $S^{2m,alg}_R \times S^{2n,alg}_R$ to the identification simplicial scheme
$$(\Spec R) \cup_{\bA_R^{m+n}{\text-}\{ 0 \} \times \{ (0,0),(0,1),(1,0) \}  }
(\bA_R^{m+n}{\text-}\{ 0 \} \times \Delta[1]\times \Delta[1])  \cup_{\bA_R^{m+n}{\text-}\{ 0 \} \times\{(1,1) \}}   
\bA_R^{m+n}$$
composed with the map to $S^{2m+2n,alg}_R$ induced by the simplicial map $\Delta[1] \times \Delta[1] \to \Delta[1]$
sending the vertices $(0,0),(0,1),(1,0)$ of $\Delta[1] \times \Delta[1]$ to the vertex 0 of the target 
$\Delta[1]$ and sending the vertex $(1,1)$ to the vertex 1 of the target.
\end{defn}

\vskip .1in

\begin{prop}
\label{prop:alg-smash}
There is a natural $GL_{n,R}$ action on $S^{2n,alg}_R$ with the property that
$\wedge: S^{2m,alg}_R \times S^{2n,alg}_R \ \to \ S^{2m+2n,alg}_R$ is 
$GL_{m,R} \times GL_{n,R}$-equivariant and induces an isomorphism
\begin{equation}
\label{eqn:et-iso}
H^*_{et}(S^{2m+2n,alg}_R,\bZ/\ell) \ \stackrel{\sim}{\to} \ 
H^*_{et}(S^{2m,alg}_R \wedge S^{2n,alg}_R,\bZ/\ell).
\end{equation}
\end{prop}

\begin{proof}
The $GL$-equivariance follows from the fact that the $GL$-action does not affect the
simplicial coordinates.

The isomorphism (\ref{eqn:et-iso}) follows from the computation of 
$H^*_{et}(\bA^n{\text-}\{0\},\bZ/\ell)$ using the  proof of Corollary \ref{cor:Frob} 
and the spectral sequence for the cohomology of a simplical
scheme.
\end{proof}

\vskip .1in

The next proposition addresses the issue of the compatibility of the 
smash product on spheres $S^{2n}$ with the smash product of spaces 
$|S_\bC^{alg,2n}|$ closely  related to $S^{2n,alg}_\bC$.   The most natural way
to introduce such compatibility is to view $S^{2n}$ as the 1-point compactification 
$(\bC^n)^+$ of $\bC^n$ equipped with its analytic topology
and use the action of $GL_n(\bC)$ on $(\bC^n)^+$; in this case, we view 
the smash product as the map on one point compactifications, 
$(\bC^m)^+ \times (\bC^n)^+ \ \to \ (\bC^{m+n})^+$,  
induced by $\bigoplus: \bC^m \times \bC^n \to \bC^{m+n}$.

\begin{prop}
\label{prop:map-cone}
We denote by $|S^{2n}_\bC|$ the total space of simplicial topological space obtained
by applying the analytic topology functor to the simplicial scheme $S^{2n,alg}_\bC$
over $\bC$.

There are natural, $U_n$-equivariant, pointed homotopy equivalences
\begin{equation}
\label{eqn:cone-maps}
(\bC^n)^+ \ \to \ D^{2n}/S^{2n-1} \ \leftarrow cone(S^{2n-1} \to D^{2n}) \ \to
|S^{2n}_\bC|
\end{equation}
from the one-point compactification of $\bC^n$ (with its analytic topology) to the quotient
of the unit disk $D^{2n}$ in $\bC^n$ with its boundary identified to a point, and from the mapping cone
of $S^{2n-1} \to D^{2n}$ to  both $D^{2n}/S^{2n-1}$ and $|S^{2n}_\bC|$.  
Here, $U_n \subset GL_n(\bC)$ is the Lie subgroup
of unitary $n\times n$ matrices.

Moreover, each of these maps commutes with smash products, where the smash product
$|S^{2m}_\bC| \times |S^{2n}_\bC| \ \to \ |S^{2m+n}_\bC|$ 
is determined by the ``algebraic smash product" of Definition \ref{defn:smash-alg}, so that the
compatibility for smash products for the right map of (\ref{eqn:cone-maps})  involves 
the composition
$$cone(S^{2m-1} \to D^{2n})\wedge cone(S^{2n-1} \to D^{2n}) \ \to \ |S^{2m}_\bC| \wedge |S^{2n}_\bC|
\ \to \ |S^{2m+2n}_\bC|.
$$
\end{prop}

\begin{proof}
The first map of (\ref{eqn:cone-maps}) is the homeomorphism given by sending $z \in \bC^n$ to 
$(1/||z||+1)\cdot z$; the second is the homotopy equivalence given by collapsing the cone
on $S^{2n-1}$ to a point; the third is the homeomorphism determined by construction of 
the total space of a simplicial space.

The compatibility of smash products is elementary topology. 
\end{proof}

\vskip .1in

We shall use the ``simplicial bar construction" for topological groups and algebraic groups.  For
the reader's convenience, we specify our notation for these simplicial objects.

\begin{notation}
\label{note:bar}
For a  closed subgroup scheme $\bG_R$ of $GL_{N,R}$ for some 
$N$ (we shall take $R$ to be $\bC$ or 
$W(k)$ or $k$), we denote by $B\bG_R$ the simplicial $R$-scheme obtained by
applying the ``simplicial bar construction"; so defined, $B\bG_R$ has $\Spec R$ in simplicial 
degree 0 and $\bG_R^{\times n}$ in simplicial degree $n > 0$; 
face maps $\bG_R^{\times n} \to \bG_R^{\times n-1}$ are determined by multiplications of adjacent 
copies of $\bG_R$ as well as projections; degeneracy 
maps $\bG_R^{\times n} \to \bG_R^{\times n+1}$ involve projections 
and the identity map $\Spec R \to \bG_R$.   If $X_\bu$ is a simplicial $R$-scheme
with $\bG_R$ acting on each $X_n$ and with each structure map of $X_n \to X_m$ a map of 
$\bG_R$-schemes, then we define $B(\bG_R,X_\bu)$ to be the diagonal of the evident 
bisimplicial $R$-scheme:   $B(\bG_R,X_\bu)_0 = X_0$, \ $B(\bG_R,X_\bu)_n \ = \ 
\bG_k^{\times n}\times_{\Spec R} X_n$, and either the first (or last, depending upon convention
chosen) face map of $B(\bG_R,X_\bu)$ entails the action $\bG_R \times X_n \to X_{n-1}$.

In order to use a parallel construction for Lie groups $G$ acting continuously 
on a simplicial topological 
space $T$, we use the simplicial bar constructions to define $BG$ and $B(G,T)$ as the
total spaces of the simplicial spaces $BG_\bu$ and $B(G,T)_\bu$; here 
$(BG)_0$ is a point and $(BG)_n$ the topological space $G^{\times n}$ for $n > 0$;
$B(G,T)_0$ is the space $T_0$ and $B(G,T)_n$ is the space $G^{\times n} \times T_n$; the face and
degeneracy maps are defined as above.

In what follows, we shall suppess the designation $(-)_\bu$ for simplicial schemes and
simplicial spaces.
\end{notation}

\vskip .1in

The relevance of the preceding corollary to our investigation of sphere fibrations
begins to appear in the following observation.

\begin{prop}
\label{prop:geom-fiber}
Consider the commutative square of simplicial $k$-varieties
\begin{equation}
\label{eqn:S-ell}
\xymatrix{
B(GL_{n,k},S_k^{2n,alg}) \ar[d]_{\tau_n^{alg}} \ar[r]^F  & B(GL_{n,k},S_k^{2n,alg}) \ar[d]^{\tau_n^{alg}} \\
BGL_{n,k} \ar[r]_F & BGL_{n,k}
}
\end{equation}
whose vertical maps are the natural projections.
Then the restriction of \\ $F: B(GL_{,k},S_k^{2n,alg}) \to B(GL_{n,k},S_k^{2n,alg})$ above
the base point $\Spec k \to BGL_{n,k}$ is the Frobenius map $F: S_k^{2n,alg} \to S_k^{2n,alg}$.
\end{prop}

\vskip .2in


\section{$\bZ/\ell$-completions and rigid \'etale homotopy types}

We begin by specifying the $\bZ/\ell$-completions  we shall use.  Throughout,
$\ell$ will denote a prime number.

\begin{defn}
\label{defn:ell-complete}
We employ the $\bZ/\ell$-completion functor of \cite[X.4.9,X.4.10]{Bo-Kan}
$$(\bZ/\ell_\infty)(-): (s.sets_*) \ \quad \to \quad (\text{Kan cxes}_*)$$
which has the important property that $(\bZ/\ell)_\infty(f): (\bZ/\ell)_\infty(X) \to 
(\bZ/\ell)_\infty(Y)$ is a homotopy equivalence whenever the map $f:X \to Y$
of simplicial sets induces an isomorphism $H_*(X,\bZ/\ell) \ \stackrel{\sim}{\to} \ H_*(Y,\bZ/\ell)$.
Another useful property of $(\bZ/\ell_\infty)(-)$ is that the natural map $(\bZ/\ell_\infty)(X\times Y) 
\to (\bZ/\ell_\infty)(X) \times (\bZ/\ell_\infty)(Y)$ is a homotopy equivalence with
a natural left inverse \cite[I.7.2]{Bo-Kan}.

To simplify notation, we shall typically abbreviate the name of the functor 
$(\bZ/\ell_\infty)(-)$ by setting
$$(-)_\ell \ \equiv \ (\bZ/\ell_\infty)(-): (s.sets_*) \quad \to \quad (\text{Kan cxes}_*).$$ 
\end{defn}

\vskip .1in

Using the properties of $(\bZ/\ell_\infty)(-)$ mentioned in Definition \ref{defn:ell-complete},  
we easily verify that the smash product $S^{2m} \times S^{2n} \ \to \ S^{2m+2n}$ of spheres
(for example, represented by the map of one-point compactifications $(\bC^m)^+ \times (\bC^n)^+
\to (\bC^{m+n})^+$) induces a smash product natural with respect to $m$ and $n$ 
\begin{equation}
\label{eqn:wedge}
(S^{2m})_\ell \times (S^{2n})_\ell \quad \to \quad (S^{2m}\times S^{2n})_\ell
\quad \to \quad (S^{2m+2n})_\ell
\end{equation}
which induces a mod-$\ell$ equivalence \ $(S^{2m})_\ell \times (S^{2n})_\ell  \ \to \  (S^{2m}\times S^{2n})_\ell.$

\vskip .1in

We recall the $\bZ/\ell$-completed rigid \'etale homotopy type functor as 
utilized in \cite{F80},
$$(-)^\wedge \quad \equiv \quad \underset{\longleftarrow}{holim}(-)\circ 
(\bZ/\ell)_\infty(-) \circ (-)_{ret},$$
which can be viewed as a refinement of mod-$\ell$ \'etale cohomology $H^*_{et}(-,\bZ/\ell)$.  
This functor, more elaborate than the original formulation of \'etale homotopy by 
M. Artin and B. Mazur in \cite{A-M}, is designed to be more ``rigid";  its target is the category 
of simplicial sets rather than the homotopy category of simplicial sets.  We usually
restrict our consideration of schemes to those schemes of finite type over a complete discrete
valuation ring (typically, an algebraically closed field such as $k$ or $\bC$ or the Witt
vectors $W(k)$).

\begin{defn}
\label{defn:wedge}
As discussed in \cite[\S 6]{F82}, the rigid \'etale homotopy type functor 
$$(-)_{ret}: (\text{pointed simplicial schemes}) \quad \to \quad (pro{\text-}bi{\text-}s.sets_*)$$
is a functor from the category of geometrically pointed simplicial schemes to the category 
of inverse systems of pointed bi-simplicial sets.  (See Appendix B of \cite{Bo-F} for
details about bi-simplicial sets.)
For a pointed simplicial scheme $X$, we define 
$$ (X)^\wedge \ \equiv \quad (\underset{\longleftarrow}{holim}(-)\circ 
(\bZ/\ell)_\infty(-)\circ diag(-) \circ (-)_{ret})(X) \ \in \ (s.sets_*),$$
where $\underset{\longleftarrow}{holim}(-): (pro-s.sets_*) \quad \to \quad (s.sets_*)$ is the
Bousfield-Kan homotopy (inverse) limit functor \cite[XI.3.2]{Bo-Kan}.

For a pointed simplicial scheme $X$ of finite type over $\bC$, we denote by
$Sin(X(\bC))$ or $Sin \circ X(\bC)$ \  the diagonal of the bi-simplical set 
$(s,t) \mapsto  Sin_t(X_s(\bC)$, where 
$Sin(X_s(\bC))$ is the singular complex of 
the topological space of complex points of $X_s$ equipped with the 
analytic topology.
%
\end{defn}

\vskip .1in

We find it useful to observe that $(X)^\wedge$ is always a Kan complex, since 
$(\bZ/\ell)_\infty(-)$ takes values in Kan complexes and $\underset{\longleftarrow}{holim}(-)$
applied to a diagram of Kan complexes is again a Kan complex by \cite[XI.5.2]{Bo-Kan}.

	Properties of $(\bZ/\ell)_\infty(-)$ and $\underset{\longleftarrow}{holim}(-)$ enable 
the following proposition.  The input to this theorem is the comparison of the \`etale topology
and the analytic topology on a complex algebraic variety leading to the ``classical 
comparison theorem" in cohomology of Artin and Grothendieck \cite{SGA4}, subsequently
sharpened to pro-finite homotopy types by Artin and Mazur in \cite{A-M}

\vskip .1in

\begin{prop} (\cite[Cor8.5]{F82})
\label{prop:holim} 
Let $X_\bC$ be a pointed, connected simplicial scheme of finite type over $\bC$.  Then there
is a homotopy equivalence
\begin{equation}
\label{eqn:Phi_X} 
\phi_X: X(\bC)_\ell \quad \stackrel{\sim}{\to} \quad (X_\bC)^\wedge
\end{equation}
natural with respect to pointed maps of simplicial schemes over $\bC$.
\end{prop}

\vskip .1in

Fix an embedding $W(k) \hookrightarrow \bC$, where $W(k)$ denotes the Witt
vectors of $k$.  This embedding, together the
quotient map $W(k) \to k$, determines distinguished geometric points $\Spec k 
\to \Spec W(k)$, \ $\Spec \bC \to W(k)$.
From the point of view of the \`etale topology, $\Spec W(k)$ is contractible.
For many smooth (simplicial) schemes $X_R$ over $R$, the base change maps
$X_\bC \to X_{W(k)}, \ X_k \to X_{W(k)}$ induce isomorphisms in \'etale cohomology
\begin{equation}
\label{eqn:base-change}
H^*_{et}(X_\bC,\bZ/\ell) \  \stackrel{\sim}{\leftarrow} \ H^*_{et}(X_{W(k)},\bZ/\ell) \ 
 \  \stackrel{\sim}{\to} \ H^*_{et}(X_k,\bZ/\ell).
 \end{equation}
 See \cite{Deligne} and \cite{Milne}.
 For example, any scheme $X_\bZ$ which is proper and smooth over $\Spec \bZ$
 enables such base changes isomorphisms, as 
 well as the complement in such a proper, smooth scheme of a divisor with normal
 crossings; another important example for us is a reductive group scheme $\bG_{\bZ}$ over
 $\Spec \bZ$ and the bar construction $B\bG_{\bZ}$ applied to such a reductive
 group scheme.   See, for example, \cite{F82}, \cite{F-Par}.

\vskip .1in

As a consequence of the base changes isomorphisms of (\ref{eqn:base-change}),
one obtains homotopy equivalences of completed \'etale homotopy types.

\begin{prop}
\label{prop:comparison}
Equip $W(k)$ with an embedding into $\bC$.
As shown in \cite{F82}, for simplicial schemes $X_{W(k)}$ over $W(k)$ which
satisfy the etale cohomological base change isomorphisms of (\ref{eqn:base-change}),
base change determines homotopy equivalences
\begin{equation}
\label{eqn:natural} 
(X_\bC)^\wedge  \quad \stackrel{\approx}{\to} \quad
 (X_{W(k)})^\wedge  \quad \stackrel{\approx}{\leftarrow}  \quad (X_k)^\wedge.
 \end{equation}
 
 These equivalences are natural with respect to maps $X_{W(k)} \to Y_{W(k)}$ 
defined over $\Spec W(k)$.
\end{prop}

\vskip .1in

One immediate consequence of Propositions \ref{prop:holim} and \ref{prop:comparison} 
is the following assertion concerning 
the behavior of $(-)^\wedge$ with respect
to products of simplicial varieties.   This property enables us to conclude that applying
$(-)^\wedge$ to certain $\cF$-objects of simplicial schemes yields $\cF$-spaces which are special.

\begin{cor}
\label{cor:product}
Consider pointed
simplicial schemes $X_{W(k)}, \ Y_{W(k)}$ over $\Spec W(k)$ which satisfy the
base change equivalences of (\ref{eqn:natural}).  Assume that $X \times Y$ satisfies
the condition that the natural map 
$H^*(X(\bC),\bZ/\ell) \otimes H^*(Y(\bC),\bZ/\ell)  \ \to \\ H^*(X(\bC) \times Y(\bC),\bZ/\ell)$
is an isomorphism.  

Then  the natural map 
$$(X_R \times  Y_R)^\wedge \ \to \ (X_R)^\wedge \times (Y_R)^\wedge$$
is a pointed homotopy equivalence whenever $R$ equals $k, \ \bC,$ or $W(k)$.   
\end{cor}

\vskip .1in

We interpret the following proposition as asserting that the Frobenius map $F$ induces
a self-map of the $\ell$-completed sphere fibration $(\tau^{alg}_n)^\wedge$ over $(BGL_{n,k})^\wedge$
which restricts to a map of degree $p^n$ on homotopy fibers.

\begin{prop}
\label{prop:hom-fiber}
There is a natural homotopy equivalence from $ (S_k^{2n,alg})^\wedge$ to the 
homotopy fiber of the map obtained by applying $(-)^\wedge$ to the map $\tau_{n,k}^{alg}$:
$$(\tau^{alg}_{n,k})^\wedge: (B(GL_{n,k},S_k^{2n,alg}))^\wedge \ \to \ (BGL_{n,k})^\wedge.$$ 

The commutative square (\ref{eqn:S-ell}) of simplicial $k$-varieties determines
\ $F^\wedge: (\tau^{alg}_{n,k})^\wedge \\ \to \ F^{\wedge*}(\tau^{alg}_{n,k})^\wedge)$ \ whose map on 
homotopy fibers is homotopy equivalent to 
$F^\wedge: (S_k^{2n,alg})^\wedge \ \to \ (S_k^{2n,alg})^\wedge$
which has degree $p^n$.
\end{prop}

\begin{proof}
By Proposition \ref{prop:geom-fiber}, $S_k^{2n,alg}$ is the geometric fibre of $\tau_{n,k}$.
By \cite[Thm10.7]{F82}, there is a natural $\bZ/\ell$-equivalence from 
$(S_k^{2n,alg})^\wedge$ (the ``homotopy
type of the geometric fiber") to the homotopy fiber of $(\tau_{n,k}^{alg})^\wedge$ (the 
``homotopy-theoretic 
fiber").

The second assertion follows from the naturality of this comparison of geometric and 
homotopy-theoretic
fibers, plus Corollary \ref{cor:Frob} and Proposition \ref{prop:geom-fiber}.
\end{proof}

\vskip .2in


\section{Some aspects of $\cF$-spaces}
\label{sec:aspects}

We require compatible structures on the collection of $(S_k^{2n,alg})^\wedge$- fibrations 
$(\tau_{n,k}^{alg})^\wedge$ of Proposition \ref{prop:hom-fiber} 
in order to specify a homotopy class of spectra from the spectrum associated to
 $\langle \{  (BGL_{n,k})^\wedge, n\geq 0 \} , \ \bigoplus\rangle$
to the spectrum associated to $\langle\{  BG_\ell(S^{2n}),n\geq 0 \}, \ \wedge \rangle$.  
Here, $G_\ell(S^{2n})$ is the simplicial monoid of $\bZ/\ell$-equivalences
\ $X \to  (\bZ/\ell)_\infty(X) \equiv X_\ell$ 
from some simplicial model $X$ of the $2n$-sphere $X$
(i.e., the geometric realization of $X$ is homotopy equivalent to $S^{2n}$)  to  $X_\ell$. 

We remind the reader that the category $\cF$ (opposite to Segal's category $\Gamma$)
has objects consisting of the finite pointed sets ${\bf n} = \{ 0,1,\dots,n\}$ and maps 
${\bf m} \to {\bf n} $ which are pointed (i.e., 0-preserving) maps of sets.  A functor from
$\cF$ to a category $\cC$ equipped with a given final-cofinal object which sends
 ${\bf 0}$ to this final-cofinal object is called an $\cF$-object of $\cC$.
Of particular interest is the case in which $\cC$ is the category of pointed simplicial sets; 
in this case, we call such a functor an $\cF$-space.

In \cite{Bo-F}, the category of $\cF$-spaces is provided with the structure of 
a Quillen model category.  We omit explicit discussion of model-theoretic 
aspects (which can be found in \cite{Bo-F}, \cite{F80})
such as fibrant replacement of objects/maps and cofibrant replacement
of objects in the category of $\cF$-spaces.   This Quillen model category 
structure formalizes the concept of homotopy
equivalent maps of $\cF$-spaces and homotopy equivalent $\cF$-spaces.
In \cite{F80}, various constructions involving $\cF$-spaces are implemented using the
technology developed in \cite{Bo-F}.  
%

More details about $\cF$-objects of $\cC$,  where
$\cC$ is a category arising in algebraic geometry,
can be found in \cite[Thm1.3]{FII}. 
As discussed in \cite{Segal}, one can and should consider multiplicative structures for
certain $\cF$-spaces including the ones we consider.   Doing so would enable our
conclusions to be stated in the category of ringed spectra. 

The motivation for the following theorem due to A.K. Bousfield and the author was to 
justify that arguments employed in \cite{F80} and the present paper involving $\cF$-spaces
establish desired properties for associated spectra.

\vskip .1in

\begin{thm} \cite[Thm 5.8]{Bo-F}
\label{thm:homotopy-cat}
There are natural adjoint equivalences relating the  homotopy category of $\cF$-spaces
and the  homotopy category of connective spectra.
\end{thm}

\vskip .1in

We denote the spectrum associated to an $\cF$-space ${\ul \cB}$ by $||\ul\cB||$.

\vskip .1in

We recall that an $\cF$-space ${\ul \cB}: \cF \to (s.sets_*)$ is said to be special if the product
of projection maps 
\ ${\ul \cB}(p_i): {\ul \cB}({\bf n}) \to \prod_{i=1}^n {\ul \cB}({\bf 1})$  are weak homotopy 
equivalences for each $n > 0$, where $p_i$ is the map induced by the map ${\bf  n} \to {\bf 1}$ in $\cF$
 sending $j \in \{ 0,\ldots,n \}$ to $1$ if $j=i$ and to 0 otherwise.
 As shown by G. Segal in \cite{Segal}, a special $\cF$-space
${\ul \cB}: \cF \to (s.sets_*)$ determines a connective $\Omega$-spectrum 
$||{\ul \cB}||$ (whose $n$-th term is $n-1$-connected).
If an $\cF$-space ${\ul \cB}: \cF \to (s.sets_*)$ is special, then $||{\ul \cB}||_0$ (the 0-space 
of the $\Omega$-spectrum $||{\ul \cB}||$) is a ``topological group completion" of ${\ul \cB}({\bf 1})$. 
(See \cite{Mc-Seg}.)
Indeed, the prolongation of $\ul \cB$ to a functor on finite, pointed simplicial sets 
enables the description of the $n$-th term of the $\Omega$-spectrum $||{\ul \cB}||$ for $n > 0$ 
as $\ul \cB (\Sigma^n)$, where $\Sigma^n$ is a finite simplicial set with $|\Sigma^n|$ 
homotopy equivalent to $S^n$.  The space $\ul \cB (\Sigma^n) \ = \ ||{\ul \cB}||_n$ is 
$n-1$-connected.  See \cite{Bo-F}.

\vskip .1in

\begin{cor}
\label{cor:homotopy-cat}
Two maps $\phi, \psi: {\ul \cB} \to {\ul \cB}^\prime$ of $\cF$-spaces which are 
homotopy equivalent as maps of $\cF$-spaces induce homotopy equivalent maps \ $||\phi||, \ ||\psi||: 
||{\ul \cB}||
\ \to \ ||{\ul \cB^\prime}||$ of spectra.
\end{cor}

\begin{proof}
In the proof of \cite[Thm 5.2]{Bo-F}, a functor $T: (\cF{\text-spaces}) \ \to \ (\cF{\text-spaces})$ is
constructed with the property that for every ${\ul \cB} \in (\cF{\text-spaces})$ the $\cF$-space
$T({\ul \cB})$ is very special (i.e., special and satisfies the 
condition that $\pi_0(T({\ul \cB})({\bf 1}))$ is a group).
Moreover, $T$ is equipped with a natural transformation $id \to T$ with the property that for 
every ${\ul \cB} \in (\cF{\text-spaces})$ the map ${\ul \cB} \to T({\ul \cB})$ is a 
``stable equivalence" of $\cF$-spaces.
Consequently, the corollary follows from Theorem \ref{thm:homotopy-cat}.
\end{proof}

\vskip .1in

\begin{ex}
\label{ex:cN}
We denote by $\ul \cN: \cF \to (sets_*)$  the functor sending ${\bf n}$ to the discrete
pointed simplicial set $\bN^{\times n}$ (the pointed set of $n$-tuples of non-negative integers)
and sending $\alpha: {\bf n} \to {\bf m}$ in $\cF$ to the pointed set map
$\bN^{\times n} \to \bN^{\times m}$ which sends the $n$-tuple $I=(i_1,\dots,i_n)$ to the $m$-tuple
$\alpha(I) = (j_1,\dots,j_m)$ where $j_t = \sum_{\alpha(s)= t} i_s$.

Thus, the spectrum $|| \ul\cN ||$ is given by $|| \ul\cN ||_0 = (\bN)^+ = K(\bZ,0)$ and
 \ $|| \ul\cN ||_n = K(\bZ,n)$
for $n > 0$.
\end{ex}

\vskip .1in

If $\ul\cB$ is an $\cF$-space over $\ul\cN$, then we denote by $\ul\cB_I \ \subset \ul\cB({\bf n})$
the pre-image of $I \in \ul\cN({\bf n})$.  For $n = 1$ and $I = \{ i \}$, we shall denote
$\ul\cB_{\{ i \} }$ by $\ul\cB_i$.
 
As seen in \cite{Segal} (see also \cite{May}), topological permutative categories determine $\cF$-spaces
over $\ul\cN$.

\vskip .1in

We have the evident criterion for a map ${\ul \cB} \ \to \ {\ul \cB^\prime}$ to induce a homotopy equivalence
of spectra.
This follows from a comparison of Quillen model category structures of $\cF$-spaces and on
spectra as given in \cite{Bo-Kan}.  More concretely, this can be proved using the explicit construction
given in \cite{Segal} of the spectrum  $||{\ul \cB}||$ associated to an $\cF$-space ${\ul \cB}$.

\begin{prop}
\label{prop:obvious}
If a map $\phi: {\ul \cB} \ \to \ {\ul \cB^\prime}$ of $\cF$-spaces over $\ul \cN$ 
satisfies the condition that
$\phi_I: {\ul \cB}_I \ \to \ {\ul \cB^\prime}_I$ is a weak homotopy equivalence 
for all $I \in {\ul \cN}(\bf n)$, 
then $\phi$ is a homotopy equivalence of $\cF$-spaces.
\end{prop}

\vskip .1in

In what follows, we use $X^n$ to denote the $n$-fold smash product of the pointed, connected
simplicial set $X$ (with $X^0 = \Delta[0]$, so that $|X_0|$ is a single point).

We remind the reader that $Sin(T)$ is a Kan complex for any topological space.  The
topological spaces we consider are compactly generated Hausdorff spaces.
To avoid the awkwardness that smash products of Kan complexes need not be
Kan complexes, we consider  monoids of weak equivalences $X^i \to Sin(|X^i|)$
rather than monoids of self-equivalences of Kan complexes.

\vskip .1in

\begin{defn}
\label{defn:G(X)}
Let $X$ be a pointed simplicial set with geometric realization $|X|$.  We consider the 
function complex $\ul{Hom}_*(X,Sin(|X|))$ whose $t$-simplices are simplicial maps
$X \times \Delta[t] \to Sin(|X|)$ sending $x_0 \times \Delta[t]$ to the degeneracy of the 
base point $x_0$ in $Sin_0(|X|)$.  The adjunction $| - | \circ Sin(-) \to id$ equips 
$\ul{Hom}_*(X,Sin(|X|))$ with a simplicial monoid structure.  

We denote by $G(X) \ \subset \ \ul{Hom}_*(X,Sin(|X|))$ 
the (simplicial) submonoid consisting of components of $G(X)$ whose images in 
the discrete monoid $\pi_0 (\ul{Hom}_*(X,Sin(|X|)))$ are invertible, and we
denote by  $G^o(X) \subset G(X)$ the distinguished component containing 
the canonical map $X \to Sin(|X|)$ (adjoint to the
identity of $|X|$).

The adjunction equivalence \ $\ul{Hom}_*(X,Sin(|X|)) \ \stackrel{\sim}{\to} \ Sin(\ul{Hom}_*(|X|,|X|)$
implies that $G(X)$ is naturally isomorphic to the singular complex of the 
H-space (with the compact open topology) of pointed homotopy self-equivalences of $|X|$.  (See,
for example, \cite[Thm 6.1]{May67}.)
\end{defn}

\vskip .1in

\begin{ex}
\label{ex:sphere-spectrum}
If $|X|$ has the homotopy type of the (pointed) $n$-sphere $S^n$, then $G(X)$ is the simplicial version 
of the mapping space $F(n+1)$ considered by J. F. Adams in \cite[\S2]{Adams} of base point 
preserving maps from $S^n$ to itself.  Thus, for $|X|$ having the homotopy type of $S^n$,
$G(X)$ is a simplicial model for $\Omega^n(S^n)$.
\end{ex}

\vskip .1in

 The $\cF$-space $\ul \cB(X)$ constructed below is a special $\cF$-space
by (\ref{eqn:explicitly}).

\begin{construct}
\label{construct:G(X)}  (\cite[Ex5.2]{F80}, \cite[\S 2]{Segal}))
Let $X$ be a connected, pointed simplicial set.  
Following Segal, the classifying $\cF$-space
\ $\ul\cB G(X)$\ over $\ul\cN$ is constructed with \ 
$\ul\cB G(X)({\bf 1}) \ = \ \coprod_{n\geq 0} BG(X^n)$,
where $BG(X^n)$ is the diagonal of the simplicial bar construction applied to the 
simplicial monoid $G(X^n)$.
%

We can make this explicit as follows.  The restriction of $\ul\cB G(X)$ above 
$I = (i_1,\ldots,i_n) \in \ul\cN({\bf n})$ is given by 
\begin{equation}
\label{eqn:explicitly}
(\ul \cB G(X))_I \ = \ \prod_{j=1}^n BG(X^{i_j}) \times 
\prod_{T \subset {\bf n} } B(Iso(X^{i_T}),Iso(X^{i_T})) \ ,
\end{equation}
where $T$ runs over pointed subsets of ${\bf n}$  with more than one non-zero element 
and $i_T = \sum_{0 \not= j \in T} i_j$.   Here, $Iso(X^{i_T})$ is the function complex of 
pointed automorphisms of $X^{i_T}$ (preserving base points).  
A non-decreasing map $\alpha: {\bf n} \to {\bf m}$ sending $I \in \ul\cN({\bf n})$ to 
$J \in \ul\cN({\bf m})$ determines
$(\ul \cB G(X))(\alpha): (\ul \cB G(X))_I \to (\ul \cB G(X))_J$ given by smash 
products.   

For example, taking 
$I = (i,j) \in \ul\cN({\bf 2})$ and $\alpha: {\bf 2} \to {\bf 1}$ sending
both $1, \ 2 \in  {\bf 2}$ to $1 \in {\bf 1}$, this becomes
$$(\ul \cB G(X))(\alpha): BG(X^{i}) \times BG(X^{j}) \times B(Iso(X^{i+j}),Iso(X^{i+j})) 
\quad \to \quad BG(X^{i+j})$$
given in simplicial degree $t$ by sending $(g_1,\ldots,g_t), (g_1^\prime,\ldots,g_t^\prime), (h_1,\ldots,h_{t+1})$
to $(\ldots,h_i^{-1}\circ (g_i\wedge g_i^\prime)\circ h_{i+1}, \ldots)$.
This is a map homotopy equivalent to the smash product $BG(X^{i}) \times BG(X^{j}) \ \to
\ BG(X^{i+j})$ which remains unchanged when composed with the permutation
$\sigma: (i,j) \mapsto (j,i)$.
\end{construct}

\vskip .1in

%
%
%

We make explicit the ``universal $X$-fibration" $\pi_X: \ul \cB (G,X) \to \ul \cB(X)$ constructed
in \cite[Ex5.3]{F80}.

\vskip .1in

\begin{construct}
\label{construct:X-action}  
We continue with the notation and constructions of Construction \ref{construct:G(X)}. 
Using the action of 
$G(X^n)$ on $Sin(|X^n|)$, we consider 
$$(\ul \cB (G(X),X))_I \ = \ \prod_{j=1}^n B(G(X^{i_j}),Sin(|X^{i_j}|)) \times 
\prod_{T \subset {\bf n} } B(Iso(X^{i_T}),Iso(X^{i_T}))$$
and the evident projection
$$\pi_{X,I}: (\ul \cB (G(X),X))_I \quad \to \quad (\ul \cB G(X))_I.$$
For $\alpha: \ul 2 \to \ul 1$
sending both $1, 2 \in \ul 2$ to $1 \in \ul 1$ this map is determined by
$$B(G(X^{i_1}),Sin(|X^{i_1}|)) \times B(G(X^{i_2}),Sin(|X^{i_2}|)) \quad \to \quad
B(G(X^{i_1+i_2}),Sin(|X^{i_1+i_2}|))$$
given by the natural factorization of the smash product $Sin(|X^{i_1}|) \times Sin(|X^{i_2}|)
\to Sin(|X^{i_1}|)\wedge Sin(|X^{i_2}|)$ through $Sin(|X^{i_1+i_2}|)$.
For any $I \in \ul\cN(\ul n)$ and any $\alpha: {\ul n} \to {\ul m} \in \cF$ sending $I$
to $J \in \ul\cN(\ul m)$, we define $\pi_{X,I} \to \pi_{X,J}$ extending this pairing.
Consequently, the maps $\pi_{X,I}$ determine a map of special $\cF$-spaces
\begin{equation}
\label{eqn:pi-univ}
\pi_X: \ul \cB (G(X),X) \quad \to \quad \ul \cB G(X).
\end{equation}
%
%
\end{construct}

\vskip .1in

\begin{remark}
\label{rem:orient}
Assuming $X$ is a pointed connected simplicial set such that permuting factors of  the smash product 
$X^n$ for each $n > 0$ is weakly homotopic to the identity, one obtains the ``oriented analogue"
\begin{equation}
\label{eqn:pi-univ-or}
\pi_{X,o}: \ul \cB (G^o(X),X) \quad \to \quad \ul \cB G^o(X).
\end{equation}
To do so, one replaces
 $G(X^n)$ by $G^o(X^n)$ 
and $Iso(X^{i_T})$ by $Iso^o(X^{i_T})$ in the preceding constructions.
\end{remark}

\vskip .2in


\section{$X$-fibrations over $\cF$-spaces}
\label{sec:X-fib}

We present the definition of an $X$-fibration  in Definition \ref{defn:X-fibration},
a correction of that in  \cite[Defn3.2]{F80} using a construction of 
Bhattacharya and Kitchloo in \cite[Defn 3.11]{B-K}.  This 
revision imposes a ``compatibility of 
smash product" condition for each $I \in \ul\cN(\bf n)$ given in Definition \ref{defn:X-fibration}(3), 
thereby implying that  $\alpha: \ul\cE({\bf n})_I \to \ul\cE({\bf m})_{\alpha(I)}$ induces the
designated smash product on fibers over $\alpha: \ul\cB({\bf n})_I \to \ul\cB({\bf m})_{\alpha(I)}$
for any $\alpha: {\bf n} \ \to \ {\bf m} \ in \ \cF$.  The requirement 
of Definition \ref{defn:X-fibration} that $f: \cE \to \cB$ is equipped with a Reedy sectioning 
as in Definition \ref{defn:Reedy} is required for the 
``principalization" construction of Construction \ref{construct:principal-ell-spaces} 

The main result of this section is Theorem \ref{thm:X-universal} which tells us that the functor
sending an $\cF$-space $\ul\cB$ over $\ul\cN$ to the set of of fiber homotopy equivalence
classes of $X$-fibrations over $\ul\cB$ is representable, given by pull-backs along 
homotopy classes of maps $f: \ul\cB \to \ul\cB G(X)$ of the universal $X$-fibration
of (\ref{eqn:pi-univ})
\vskip .1in

\begin{defn} \cite[Defn 3.11]{B-K}
\label{defn:Reedy}
Let $f: \ul\cE \ \to \ \ul\cB$ be a map of special $\cF$-spaces of $\ul\cN$.  A Reedy
sectioning of $f$ is the choice for eachl $n > 0, \ I=(i_1,\ldots,i_n) \in \cF({\bf n})$, and
subset $S \subset I$ of a subspace \ $\cE_{I,S} \subset \cE_I$ \ satisfying 
the following conditions:
\begin{enumerate}
\item
If $S = I$, then  $\cE_{I,S} \ = \  \cE_I$; \ if $S = \emptyset$, then  $\cE_{I,S} \ = \  \cB_I$.
\item
For a given $I$ and varying subsets $S \subset I$, the embedding  
$\cE_{I,S} \subset \cE_I$ is natural with respect to inclusions and
retractions of subsets.
\item
For $\alpha: {\bf n} \to {\bf m}$, set $\alpha_I(S) \ \subset \ \alpha(I) = (j_1,\ldots,j_m) \in \cF({\bf m})$
equal to the subset of $\alpha(I)$ of those $j_t \equiv \sum_{\alpha(s)=t} i_s$ such that 
every $s \in I$ with $\alpha(s) = t$ lies in $S$.  Then $\alpha: \ul\cE_I \to \ul\cE_{\alpha(I)}$
restricts to $\ul\cE_{I,S} \ \to \ \ul\cE_{\alpha_I,\alpha(S)}$.  

\end{enumerate}
\end{defn}

We define an $X$-fibration over $\ul\cB$ revising the definition of \cite[Defn 3.2]{F80}.  

\begin{defn}
\label{defn:X-fibration}
Let $X$ be a pointed, connected simplicial set.
An $X$-fibration 
is a map  \ $f: \ul\cE \ \to \ \ul\cB$ \ of special $\cF$-spaces 
over $\ul \cN$ equipped with a Reedy sectioning
satisfying the following conditions:
\begin{enumerate}
\item
For each $n > 0, \ I \in \cF({\bf n})$, $f_I: \ul\cE_I\to \ul\cB_I $ is a fibration with fibers which
are pointed homotopy equivalent to $Sin(|X^I|)$.
\item
For any ${\bf  n} \in \cF$ and $I = (i_1,\ldots,i_s, \ldots i_n) \in \ul\cN({\bf n})$, the projection maps
maps $p_i: {\bf n} \to {\bf 1}$ (sending $s \not= i$ to 0 and sending $i$ to 1) induce a weak homotopy
equivalence $\rho_I: \ul \cE_I \to \prod_{s=1}^n {\ul\cE}_{i_s} \times_{ \prod {\ul\cB}_{i_s} } {\ul\cB}_I$.
\item
For any $\alpha: \bf n \to \bf m$ in $\cF$,  $I = (i_1,\ldots,i_n)\in \ul\cN(\bf n)$
and any $0 \not= t \in \bf m$,
let $I^{-1}(t) \subset \ul\cN(\bf r)$ denote the ordered subset of $I$  consisting of those entries $i_s$ of $I$
such that $\alpha(s) = t$.   Denote by $\beta: \bf n \to \bf r$ the surjective map
sending $i_s$ to 0 if $i_s \notin I^{-1}(t)$ and otherwise preserving the ordering.  Denote by
$\mu: \bf{r} \to \bf 1$ the map sending each $0 \not= i \in  \bf r$ to $1 \in \bf 1$.
Consider the commutative square
\begin{equation}
\begin{xy}*!C\xybox{%
\xymatrix{  \ul\cE_I \ar[r]^-{\alpha} \ar[d]^{\beta} & \ul\cE_{I^{-1}(t)} \ar[d]^{p_t} \\
\ul\cE_{I^{-1}(t)} \ar[r]^-{\mu} & \ul\cE_t \ .
 }
}\end{xy}
\end{equation}
We require that $\mu:  \ul\cE_{I^{-1}(t)} \ \to \ul \cE_t$  restricted to fibers over 
$\ul\cB_{I^{-1}(t)}   \stackrel{\mu}{\to} \ul\cB_t$ is
homotopy equivalent to the  smash product  $Sin| X^{I^{-1}(t) }|) \ \to \ Sin|X^t|.$
\end{enumerate}

If $X$ satisfies the condition that permuting factors of  the smash product 
$X^n$ for each $n > 0$ is weakly homotopic to the identity and if $f: \ul\cE \to \ul\cB$ is an
$X$ fibration, then we say that $f$ is an oriented $X$-fibration if the pointed homotopy equivalences 
of (1) are all homotopy equivalent to the appropriate identities and the weak homotopy equivalences
$\rho_I$ of (2) restrict on fibers to maps which are homotopy equivalent  to the identity.
\end{defn}

\vskip .1in
The role of Reedy sectioning of an $X$-fibration becomes apparent with the following
basic example.

\begin{ex}
\label{ex:Reedy}
Set $\ul\cB \ = \ \ul\cN$ and define $f: \ul\cE \ \to \ \ul\cN$ by setting $\ul\cE_I = X^I \equiv X^{i_1} \times \cdots \times X^{i_n}$
for any $n > 0, \ I = (i_1,\ldots,i_n) \in \cF({\bf n})$. For $S = \{t_1\ldots, t_s\} \subset I$, define $\ul\cE_{I,S} \hookrightarrow \ul\cE_I$ 
to be the inclusion of the partial product $X_{I,S} \equiv X^{t_1} \times \cdots \times X^{t_s} \hookrightarrow X^I$.  
If $\alpha: {\bf n} \to {\bf m}$ in $\cF$ sends
$I$ to $J = (j_1,\ldots,j_m)$ with $j_t = \sum_{s:\alpha(s) = t} i_s$ and if $S_{j_t} = \{s:\alpha(s) = t\} \subset I$, then 
$\alpha: X^I \ \to \ X^{\alpha(I)}$ sends each $X_{I_,S_t} \subset X^I$ to the base point of $X^{\alpha(I)}$.
\end{ex}

\vskip .1in

\begin{ex}
\label{ex:pi-X}
For any pointed, connected simplicial set $X$,  
the  map of $\cF$-spaces of (\ref{eqn:pi-univ})
is an $X$-fibration (which we call the ``universal $X$-fibration"). 

Similarly, the map of $\cF$-spaces of (\ref{eqn:pi-univ-or}) is an oriented $X$-fibration
(which we call the ``universal oriented $X$-fibration"). 
\end{ex}

\vskip .1in

\begin{defn}
\label{defn:X-map}
Let $X$ be a pointed, connected simplicial set.
Consider a map $\phi: \ul\cB \ \to \ \ul \cB^\prime$ of special $\cF$-spaces over
$\ul\cN$, and $X$-fibrations $f: \ul\cE \to \ul \cB, \ f^\prime: \ul\cE ^\prime\to \ul \cB^\prime$
equipped with given Reedy sections.
A map $\phi: f \ \to \ f^\prime$  of $X$-fibrations over $\phi$ consists of a commutative square of
$\cF$-spaces over $\ul\cN$ restricting to Reedy sections which restricts to fiber homotopy equivalences
\begin{equation}
\label{eqn:pull-fib_i}
\xymatrix{
({\ul\cE})_I  \ar[r]^{\tilde \phi_I}  \ar[d]_{f_I}  &  ({\ul\cE^\prime})_I \ar[d]^{f^\prime_I} \\
({\ul\cB})_I \ar[r]_{\phi_I }& ({\ul\cB^\prime})_I
}
\end{equation}
for all $n > 0, \ I \in \cF(\bf n)$.

Two $X$-fibrations $f: \ul\cE \to \ul\cB,  \ f^\prime: \ul\cE^\prime \to \ul\cB^\prime$  
are said to be fiber homotopy equivalent if there exists a chain of pointed weak homotopy equivalences of $\cF$-spaces
\begin{equation}
\label{eqn:chain1}
\ul\cB \ \stackrel{\phi_1}{\leftarrow}  \ \ul\cB_1 \ \stackrel{\phi_2}{\to} \ \ul\cB_2 \  
\stackrel{\phi_3}{ \leftarrow} \quad  \cdots \quad \ul\cB_{2n-1} \ \stackrel{\phi_{2n}}{\to} \ \ul\cB 
\end{equation}
which are covered by a chain of maps of $X$-fibrations 
\begin{equation}
\label{eqn:chain2}
f \ \stackrel{\tilde\phi_1}{\leftarrow}  \ f_1 \ \stackrel{\tilde \phi_2}{\to} \ f_2 \  
\stackrel{\tilde\phi_3}{ \leftarrow} \quad  \cdots \quad f_{2n-1} \ \stackrel{\tilde \phi_{2n}}{\to} \ f^\prime
\end{equation}
with each $\tilde \phi_i$ restricting to homotopy equivalences on fibers.

If both $f: \ul\cE \to \ul \cB, \ f^\prime: \ul\cE ^\prime\to \ul \cB^\prime$ are oriented $X$-fibrations,
then we say that they are oriented fiber homotopy equivalent 
 if they are related by  a chain of oriented maps of $X$- fibrations 
as in (\ref{eqn:chain2}) covering a chain of weak homotopy equivalences as in (\ref{eqn:chain1}).
\end{defn}

 \vskip .1in

\begin{remark}
\label{rem:section}
The structure of an $X$-fibration of $\cF$-spaces $f: \ul\cE \to \ul\cB$ with section $\eta:\ul\cB \to \ul\cE$
differs from the structure of a ``sectioned map" as in \cite[Defn2.3]{F80}.
Namely, Definition \ref{defn:X-fibration}(3) is a condition given for each $I \in {\bf N}^n$ .
Functoriality of our constructions guarantees that this condition  is satisfied by our examples.
\end{remark}

\vskip .1in

The following definition of principal $d$-simplices is essentially that of \cite[Defn 4.1]{F80}.

\begin{defn}
\label{defn:principal-simplex}
Let $X$ be a pointed, connected simplicial set.
Consider  an $X$-fibration $f: \ul\cE \ \to \ \ul\cB$.
 For $I = (i_,\ldots,i_n) \in \ul\cN({\bf n})$, we say that a map $\phi: X^I \times \Delta[d]
 \to Sin(|\ul\cE_I |)$ is a principal $d$-simplex of $f$ provided that
 \begin{enumerate}
 \item
 The composition $f_I \circ \phi: X^I \times \Delta[d] \to Sin(|\ul\cE_I |) \to Sin(|\ul\cB_I |)$ factors though the
 projection, $X^I \times \Delta[d] \to \Delta[d] \to Sin(|\ul\cB_I |)$, and determines a weak homotopy equivalence
 $X^I \times \Delta[d] \ \to \ Sin(|\ul\cE_I |) \times_{Sin(|\ul\cB_I |)} \Delta[d]$.
 \item
 For every surjective map $\alpha: {\bf n} \to {\bf r}$ in $\cF$,  set $\alpha(I) = (j_1,\ldots,j_r) \in \ul\cN(r), \ j_s = \sum_{\alpha(i) = s} i_s$.
 Then $\phi$ induces $\phi^\alpha: X^{\alpha(I)} \times \Delta[d] \to Sin(|\ul\cE_{\alpha(I)} |)$ fitting in the commutative square
 \begin{equation}
\xymatrix{
X^I \times \Delta[d] \ar[d]_{\alpha} \ar[r]^-\phi &  Sin(|\ul\cE_I|) \ar[d]^{\alpha} \\
X^{\alpha(I)} \times \Delta[d] \ar[r]^-{\phi^\alpha} & Sin(|\ul\cE_{\alpha(I)})|
}
\end{equation}
such that $\phi^\alpha$ induces a weak homotopy equivalence 
$$X^{\alpha(I)} \times \Delta[d] \ \to \ Sin(|\ul\cE_{\alpha(I)} |) \times_{Sin(|\ul\cB_{\alpha(I)} )} \Delta[d].$$
 \end{enumerate}.
\end{defn}

\vskip .1in

Using the natural retraction of $Sin( || Sin ( ||-|| ))) \quad \to \quad Sin(||-||)$
of the canonical inclusion, we obtain the following natural action of $G(X^I)$
on principal simplicies.

\begin{lemma}
\label{lem:action-Sin}
Let $X$ be a pointed, connected simplicial set and $f: \ul\cE \ \to \ \ul\cB$ and $X$-fibration.
For $I = (i_,\ldots,i_n) \in \ul\cN({\bf n})$, set \ $P(f_I) \ \subset \ \ul{Hom}(X^I,Sin(|\cE_I |))$ \ to be 
the subcomplex of  principal simplices.  Then there is a right action
\begin{equation}
\label{eqn:prin-act}
P(f_I) \times G(X^I) \quad \to \quad P(f_I)
\end{equation}
sending $\phi: X^I \times \Delta[d] \to Sin(|\ul\cE_I |), \theta: X^I \times \Delta[d] \to Sin(| X^I |)$ to 
the composition
$$X^I \times \Delta[d] \stackrel{\theta,pr_s}{\to} Sin(| X^I |) \times \Delta[d] \stackrel{Sin\circ |\phi|}{\to} 
Sin(|(Sin(|\cE_I |)|))   \to Sin(|\cE_I |)).$$
For any $n > 0, \ I\in \cF({\bf n})$ and any $\alpha: {\bf n} \to {\bf m}$, the folllowing square commutes
\begin{equation}
\label{eqn:prin-square}
\xymatrix{
P(f_I) \times G(X^I) \ar[r] \ar[d]_\alpha &  P(f_I) \ar[d]^\alpha \\
P(f_{\alpha(I})) \times G(X^{\alpha(I)}) \ar[r] & P(f_{\alpha(I)}) .
}
\end{equation}
By construction, $P(f_I)$ is a principal $G(X^I)$-fibration over $\ul\cB_I$.
\end{lemma}

\begin{proof}
The commutativity of (\ref{eqn:prin-square}) is implied by the Reedy sectioning of $f$.
\end{proof}

\vskip .1in

The following construction  introduces the associated 2-sided bar construction
$B(P(f_I),G(X^I),Sin(|X^I |))$ for each $I \in \ul\cN$.

\begin{construct}
\label{construct:principal}
Let $X$ be a pointed, connected simplicial set, and consider an 
$X$-fibration  $f: \ul\cE \ \to \ \ul\cB$.
 For any $I = (i_1,\ldots,i_n) \in \ul\cN$, we define $P(f_I)$ to be the sub-complex of the 
simplicial mapping complex $\ul{Hom}(X^I,Sin(|\ul\cE_I |))$ consisting of principal simplices for $f$
(as defined in Definition \ref{defn:principal-simplex}).

The map $f_I$ induces a map $P(f_I) \to \ul\cB_I$ which admits right action by 
$G(X^I) \ \equiv \ \prod_{s=1}^n G(X^{i_s})$ (see Lemma \ref{lem:action-Sin}) which covers the identity on $\ul\cB_I$
and which satisfies the condition that for every simplex $\Delta[d] \to \ul\cB_I$
there is a $G(X^I)$-equivariant weak equivalence $(X^I \times \Delta[d]) \times G(X^I)
\to P(f_1)\times_{\ul\cB_I} \Delta[d]$.  

Moreover,  there are maps
\begin{equation}
\label{eqn:2-sided}
\xymatrix{
Sin(| \ul\cE_I |) \ar[d]^{f_I}  & \ar[l]^-{\tilde p_{f,I}} B(P(f_I),G(X^I),Sin(|X^I |)) 
 \ar[d]^{\pi_{P(f),X,I}} \ar[r]^-{\tilde q_{f,I}}  &   
 B(G(X^I),Sin(|X^I |)) \ar[d]_{\pi_{X,I}} &  \\
Sin(|\ul\cB_I |)  & \ar[l]^-{p_{f,I}} B(P(f_I),G(X^I)) \ar[r]_{q_{f,I}} &  BG(X^I)
}
\end{equation}
for any $I \in \ul\cN({\bf n})$, where the right horizontal maps are projections,  the upper left map is given by
projecting to $B(P(f_I),G(X^I))$ followed by the action of Lemma \ref{lem:action-Sin}, and the lower
left map is given by projecting to $P(f_I)$ followed by projecting to $\ul\cB_I$.

The vertical maps of (\ref{eqn:2-sided}) are fibrations with fibers homotopy equivalent to $Sin|X^I |)$ 
and the left horizontal maps are homotopy equivalences because each $P(f_I) \to \ul\cB_I$ is a 
principal $G(X^I)$-fibration.
\end{construct}

\vskip .1in

The following proposition formalizes the observation that the commutative diagrams (\ref{eqn:2-sided})
 of Construction \ref{construct:principal} fit together to determine a corresponding 
 commutative diagram of $\cF$-spaces.  This may be surprising since 
 $I \mapsto P(f_I) $ is not functorial with respect to all maps $\alpha: {\bf n} \to {\bf m}$
sending $I \in \cN(({\bf n})$ to $\alpha(I) \in \cN({\bf n})$.  

\begin{prop}
\label{prop:func-P(f)}
The diagrams of (\ref{eqn:2-sided}) for all $\ul n$ and $I \in \ul\cN(\ul n)$ determine
a commutative diagram of $\cF$-spaces 
\begin{equation}
\label{eqn:2-sided-func}
\xymatrix{
\ul\cE \ar[d]_f \ar[r] &Sin(|\ul\cE|) \ar[d]^{Sin(|f|)} & \ar[l]_-{\tilde p_f} \ul\cB (P(f),G(X),X)
 \ar[d]^{\pi_{P(f),X}} \ar[r]^-{\tilde q_f} &   \ul\cB (G(X),X) \ar[d]_{\pi_X} &  \\
\ul\cB \ar[r] & Sin(|\ul\cB|)  & \ar[l]^-{p_f} \ul\cB (P(f),G(X)) \ar[r]_{q_f} &  \ul\cB G(X) \ .
}
\end{equation}

The  vertical maps of (\ref{eqn:2-sided-func}) are $X$-fibrations,  the squares 
of (\ref{eqn:2-sided-func}) constitute maps of $X$-fibrations, and the left and middle squares 
of (\ref{eqn:2-sided-func}) are fiber homotopy equivalences of $X$-fibrations.
\end{prop}

\begin{proof}
After replacing the left $G(X^I) \ \equiv \prod_{s=1}^n G(X^{i_s})$ action on $Sin(|X^I)|)$ by
the right $G(X^I)$-action on $P(f)_I$, we repeat the discussion of
Construction \ref{construct:X-action} using the maps of (\ref{eqn:2-sided})
to obtain a map of special $\cF$-spaces
 \begin{equation}
 \label{eqn:principal}
 q_f: \ul\cB(P(f),G(X)) \quad \to \quad  \ul\cB G(X)
 \end{equation}
with
$$
\ul\cB(P(f),G(X))_I \ = \ B(P(f_I),G(X^I))\times 
\prod_{T \subset {\bf n} } B(Iso(X^{2i_T}),Iso(X^{2i_T}))$$
mapping to \ $\ul \cB G(X)$ \ 
by sending $P(f_I)$ to the base point .

Using the two sided bar construction, we define $\ul\cB(P(f),G(X),X)_I$ to be
\begin{equation}
\label{eqn:2-sided-P}
 \ul\cB(P(f),G(X),X)_I \ = \ B(P(f_I),G(X^I),Sin(|X^I |)) \times 
\prod_{T \subset {\bf n} } B(Iso(X^{2i_T}),Iso(X^{2i_T}))
\end{equation}
for any $I \in \ul\cN({\bf n})$.  Then sending each $P(f_I)$ to the base point
determines 
$$\tilde q_f: \ul\cB(P(f),G(X),X) \quad \to \quad \ul\cB (G(X),X)$$
 covering $q_f$.

As for $q_f$, the maps $p_{f,I}$ (respectively, $\tilde p_{f,i}$) of (\ref{eqn:2-sided}) 
determine the maps
$p_f: \ul\cB (P(f),G(X)) \to Sin(|\ul\cB|)$ (resp., $\tilde p_f: \ul\cB (P(f),G(X),X) \to Sin(|\ul\cE|)$).

The proof that $\pi_{P(f),X}: \ul\cB(P(f),G(X),X) \to  \ul\cB(P(f),G(X))$ is an $X$-fibration,
(and in particular is a map of $\cF$-spaces), is basically the justification given for \cite[Ex5.5]{F80}.
The fact that $\tilde p_f$ and $p_f$ are homotopy equivalences of $\cF$-spaces 
follows from the assertion of Construction \ref{construct:principal} that each
$\tilde p_{f,I}$ and each $p_{f,I}$ is a homotopy equivalence.
%
%
%
The verification of the commutativity of the squares of (\ref{eqn:2-sided-func})
is straight-forward.
\end{proof}

\vskip .1in

The proof of the following theorem given in \cite{F80} proceeds along the lines of 
arguments of J.P May in  \cite{May75} for classifying maps for fibrations with given homotopy type
as fiber.  

\begin{thm} (\cite[Thm6.1]{F80}
\label{thm:X-universal}
Let $X$ be a pointed, connected simplicial set.
Then the  $X$-fibration of $\cF$-spaces of (\ref{eqn:pi-univ}),
\ $\pi_{X}: {\ul\cB}(G(X),X) \ \to \ {\ul\cB}G(X)$, \
satisfies the following universal property.
\vskip .05in
For any special $\cF$-space ${\ul \cB}$ over $\ul \cN$ and any $X$-fibration 
$f: {\ul \cE} \ \to {\ul \cB}$ over ${\ul \cB}$,
there is a unique homotopy class of maps of $\cF$-spaces $\phi: {\ul \cB} \ \to \ {\ul\cB}G(X)$
over $\ul \cN$ such that $\phi^*(\pi_{X})$ is fiber homotopy equivalent to $f$ as $X$-fibrations 
over ${\ul\cB}$.
\end{thm}

\begin{proof}
Proposition \ref{prop:func-P(f)} implies that
the $X$-fibration $f: \ul\cE \ \to \ \ul\cB$ 
is fiber homotopy equivalent to the pull-back of the universal $X$-fibration 
$\pi_{X}: {\ul\cB}(G(X),X) \ \to \ {\ul\cB}G(X)$  via the
homotopy class of the composition of a homotopy inverse for $p_f$ and the map $q_f$:
$$\ul\cB \quad \to \quad Sinc(| \ul\cB |) \quad \stackrel{p_f}{\leftarrow} \quad \ul\cB(P(f),G(X)) \quad 
\stackrel{q_f}{\to} \quad \ul \cB G(X).$$
In other words,  every $X$-fibration $f: \ul \cE \ \to \ \ul B$ is fiber homotopy equivalent to
$\phi^*(\pi_X)$ for  $\phi \ = \ q_f \circ \rho_f^{-1}: \ul\cB \ \to \ \ul \cB G(X)$.

We define an inverse to this surjective correspondence by associating 
$q_f \circ \rho_f^{-1}: \ul\cB \ \to \ul \cB G(X)$
to an $X$-fibration $f: \ul\cE \to \ul\cB$.   In order to show this is well defined on 
fiber equivalence classes of $X$-fibrations it suffices to consider 
two $X$-fibrations $f: \ul\cE \to \ul\cB, \ f^\prime: \ul\cE^\prime \to \ul\cB$ together with a
homotopy equivalence between them as in Definition \ref{defn:X-map}
\begin{equation}
\label{eqn:phi-map}
\xymatrix{
\ul\cE \ar[d]_f \ar[r]^{\tilde \phi} & \ \ul\cE^\prime \ar[d]_{f^\prime} \\
\ul\cB \ar[r]_\phi & \ul\cB^\prime \ .
}
\end{equation}
The map (\ref{eqn:phi-map} determines  a commutative diagram of $\cF$-spaces
\begin{equation}
\label{eqn:model-inverse}
\xymatrix{
& \ul\cB (P(f),G(X)) \ar[ld]_{p_f} \ar[dd]^\approx \ar[rd]^{q_f} & & \\
Sin(| \ul\cB|) & & \ul\cB G(X) \ ,\\
& \ul\cB (P(f^\prime),G(X)) \ar[lu]^{p_{f^\prime}} \ar[ru]_{q_{f^\prime}} & 
}
\end{equation}
telling us that the homotopy class of the map $q_f \circ p_f^{-1}: \ul\cB \ \to \ul \cB G(X)$
depends only the fiber homotopy equivalence class of $f$.

Finally, we verify that the composition of this inverse with our first construction is the 
identity.  This verification is given by the last part of the proof of \cite[Thm 6.1]{F80} with
only notational changes.
\end{proof}

\vskip .1in

Following Remark \ref{rem:orient}, we see that the proof of Theorem \ref{thm:X-universal}
applies to prove the following result.

\begin{cor}
\label{cor:X-universal-or}
Let $X$ be a pointed, connected simplicial set satisfying the condition that permuting factors
of the smash product is weakly homotopic to the identity for each $n > 0$.
Then the  $X$-fibration of $\cF$-spaces,
\ $\pi_{X}: {\ul\cB}(G^o(X),X) \quad \to \quad \ {\ul\cB}G^o(X)$ \ of (\ref{eqn:pi-univ})
satisfies the following universal property.
\vskip .05in
For any special $\cF$-space ${\ul \cB}$ over $\ul \cN$ and any oriented $X$-fibration 
$f: {\ul \cE} \ \to {\ul \cB}$ over ${\ul \cB}$,
there is a unique homotopy class of maps of  $\cF$-spaces $\phi: {\ul \cB} \ \to \ {\ul\cB}G^o(X)$
over $\ul \cN$ such that $\phi^*(\pi_{X,o})$ is oriented fiber homotopy equivalent to $f$ as $X$-fibrations 
over ${\ul\cB}$.
\end{cor}

\vskip .2in


\section{$\bZ/\ell$-completed $X$-fibrations}
\label{sec:ell-complete}

Our goal in this section is to provide $\bZ/\ell$-completion analogues of 
the constructions and results of Section \ref{sec:X-fib}.  We use
the $\bZ/\ell$-completion functor  $(-)_\ell \ \equiv \ (\bZ/\ell)_\infty(-)$ of 
Definition \ref{defn:ell-complete}, regularly appealing to the remarkable 
book \cite{Bo-Kan} by Bousfield and Kan.  In Definition \ref{defn:ell-fibration}, we define a
 $\bZ/\ell$-completed $X$-fibration to be a  map of $\cF$-spaces $f:\ul\cE \ \to \ \ul \cB$ over
$\ul\cN$ with the property that the fibers of $f_I: \ul\cE_I \ \to \ul \cB_I$ are homotopy equivalent
to $(\prod_{j=1}^n X^{i_j})_\ell \ \stackrel{\approx}{\to} \ \prod_{j=1}^n (X^{i_j})_\ell$ \  
 for $I = (i_1,\ldots,i_n) \in \ul\cN(\bf n)$.
The ``classifying $\cF$-space" for such fibrations involves monoids $G_\ell(X^m)$
of mod-$\ell$ equivalences $X^m \to (X^m)_\ell$.

As in Definition \ref{defn:G(X)} where we considered maps from $X$ to $Sin(| X |)$, 
it is important that we consider maps whose source is often a smash product (and thus not a Kan complex) 
and whose target is $\ell$-complete.

\vskip .1in

\begin{defn}
\label{defn:smash-ell-good}
We recall that a pointed, connected simplicial set $X$ is called $\ell$-good if the 
canonical $\bZ/\ell$-completion map $X \to X_\ell$ is a mod-$\ell$-equivalence
(i.e., induces an isomorphism in mod-$\ell$ homology).  
We say that a pointed simplicial set $X$ is ``smash $\ell$-good" if the $m$-fold smash 
product $X^m$ of $X$ is $\ell$-good for all $m > 0$.

Using the natural commutative square 
\begin{equation}
\xymatrix{
X^m \ar[d] \ar[r] &  (X^m)_\ell \ar[d] \\
(X_\ell)^m \ar[r] & ((X_\ell)^m)_\ell
}
\end{equation}
and the fact that the $(\bZ/\ell)$-completion of
$X_\ell$ of an $\ell$-good simplicial set $X$ is also $\ell$-good, we conclude that
if $X$ is smash $\ell$-good then so is $X_\ell$.
\end{defn}

\vskip .in

\begin{defn}
\label{defn:G-ell}
Let $X$ be a pointed, connected simplicial set.  Assume that $X$  is smash $\ell$-good.
For any $m > 0$, we  denote by
\begin{equation}
\label{eqn:action1}
G_\ell(X^m) \ \quad \hookrightarrow \quad \ul{Hom}(X^m,(X^m)_\ell)
\end{equation}
the function complex consisting of pointed $t$-simplices $X^m \times \Delta[t] \ \to (X^m)_\ell$
which are mod-$\ell$ equivalences.

We denote by 
\begin{equation}
\label{eqn:G-ell}
\iota: G_\ell^o(X^m) \quad \hookrightarrow \quad G_\ell(X^m) 
\end{equation}
the connected component of the distinguished map $\phi: X^m \ \to \ (X^m)_\ell$.
\end{defn}

\vskip .1in

\begin{prop}
\label{prop:G-X-ell}
The function complexes $G_\ell(X^m)$ and $G_\ell^o(X^m)$ of Definition \ref{defn:G-ell}
satisfy the following properties.
\begin{enumerate}
\item 
$G_\ell(X^m)$ is a simplicial monoid and $G_\ell^o(X^m) \ \subset \ G_\ell(X^m)$ 
is a simplicial sub-monoid, where the product structure is induced by composition of functions
and the functorial retraction $\eta: (\bZ/\ell)_\infty(-) \circ (\bZ/\ell)_\infty(-) \to (\bZ/\ell)_\infty(-)$ \cite[Lem I.5.6]{Bo-Kan}.
\item 
There is a natural action $G_\ell(X^m) \times (X^m)_\ell \ \to \  (X^m)_\ell$ determining an inclusion of simplicial 
monoids
$$G_\ell(X^m) \quad \hookrightarrow \quad \ul{Hom}((X^m)_\ell,(X^m)_\ell).$$
\item 
The action of (2) fits in 
commutative squares involving compositions and smash products
\begin{equation}
\label{eqn:ell-ell}
\xymatrix{
G_\ell(X^m) \times (X^m)_\ell \times G_\ell(X^n) \times  (X^n)_\ell \ar[d] \ar[r] &  
G_\ell(X^{m+n}) \times (X^{m+n})_\ell  \ar[d] \\
(X^m)_\ell \times  (X^v)_\ell \ar[r] & (X^{m+n})_\ell \ .
}
\end{equation}
\end{enumerate}
\end{prop}

\begin{proof}
The ``composition product" $G_\ell(X^m) \times G_\ell(X^m) \ \to \ G_\ell(X^m)$
is defined by sending $\alpha, \beta: \Delta[d] \times X^m \to (X^m)_\ell$ to the 
composition
\begin{equation}
\Delta[d] \times X^m \ \stackrel{pr_1\times \alpha}{\to} \ \Delta[d]\times (X^m)_\ell \to 
(\Delta[d]\times X^m)_\ell \ \stackrel{\beta_\ell}{\to} \ ((X^m)_\ell)_\ell \ \stackrel{\eta}{\to} \ (X^m)_\ell \ ,
\end{equation}
where $\eta: (\bZ/\ell)_\infty \circ (\bZ/\ell)_\infty \to (\bZ/\ell)_\infty$ is part of the natural 
triple structure (see, for example, \cite{MacL}). 
The action $G_\ell(X^m) \times (X^m)_\ell \ \to \  (X^m)_\ell$ is similarly defined, sending 
 
$\alpha: \Delta[d] \times X^m \to (X^m)_\ell$ to 
$$\Delta^d \times (X^m)_\ell \ \stackrel{(-)_\ell}{\to} \ (\Delta^d \times (X^m)_\ell)_\ell \ 
\stackrel{\alpha_\ell}{\to} \ ((X^m)_\ell)_\ell \ \stackrel{\eta}{\to} \ (X^m)_\ell \ .$$

If we denote this composition by $\tilde \alpha$, then we readily verify that the 
restriction of $\tilde \alpha$ to $\Delta[d] \times X^m$ equals $\alpha$.
The map $\ul{Hom}(X^m,X^m_\ell) \ \to \ \ul{Hom}((X^m)_\ell,(X^m)_\ell)$
commutes with composition thanks to the ``idempotence" of $\eta$.

The smash product 
\begin{equation}
\label{eqn:ell-prod}
G_\ell(X^m) \times G_\ell(X^n) \quad \to \quad G_\ell(X^{m+n})
\end{equation}
extending the smash product pairing  $G(X^m)\times G(X^n)\ \to \ G(X^{m+n})$
is given by sending the pair of $t$-simplices $\Delta[t] \times X^m \to  (X^m)_\ell$, \
$\Delta[t] \times X^n \to  (X^n)_\ell$ to the $t$-simplex 
$\Delta[t] \times X^{m+n} \to  (X^{m+n})_\ell$ defined as the 
factorization of the composition 
\begin{equation}
\label{eqn:comp-ell}
\Delta[t] \times X^m \times X^n \ \to  \ (X^m)_\ell \times (X^n)_\ell \
\to \ (X^m \times X^n)_\ell \  \to  \ (X^{m+n})_\ell \ ,
\end{equation}
where the second map is the canonical homotopy equivalence of \cite[I.7.2]{Bo-Kan} given
by $\bZ/\ell$-linearlity (and which is a retraction of the natural product map
$(X^m \times X^n)_\ell  \to (X^m)_\ell \times (X^n)_\ell$).  We refer to the composition
of the third and fourth maps of (\ref{eqn:comp-ell}),
\begin{equation}
\label{eqn:smash-ell}
(X^m)_\ell \times (X^n)_\ell \quad \to \quad (X^m \times X^n)_\ell \quad \to \quad (X^{m+n})_\ell \ ,
\end{equation}
as the smash product of $(X^m)_\ell$ and  $(X^n)_\ell$.

The commutativity of (\ref{eqn:ell-ell}) is left to the reader.
\end{proof}

The following definition of a $\bZ/\ell$-completed $X$-fibration 
enables considerations
of actions by $G_\ell(X^m)$ on $(X^m)_\ell$ similar to the 
actions of $G(X^m)$ on $Sin(|X^m|)$ considered in Definition \ref{defn:X-fibration}.
This definition does not mention orientations and uses simpler notation,
but is closely related to that of \cite[Defn 7.1]{F80}.

\vskip .1in

\begin{defn}
\label{defn:ell-fibration}
Let $X$ be a pointed, connected simplicial set which is smash $\ell$-good.

A $\bZ/\ell$-completed $X$-fibration over a special $\cF$-space ${\ul \cB}: \cF \to (s.sets_*)$
is a  map \ $f: \ul\cE \ \to \ \ul\cB$ \ of special $\cF$-spaces 
over $\ul \cN$ equipped with a Reedy sectioning 
satisfying the following conditions:
\begin{enumerate}
\item
For each $n > 0$ and $I \in \cF({\bf n})$, $f_I: \ul\cE_I\to \ul\cB_I $ is a fibration with fibers which
are pointed homotopy equivalent to $(X^I)_\ell$.
\item
For any $\ul n \in \cF$ and $I = (i_1,\ldots,i_s, \ldots i_n) \in \ul\cN(\bf n)$, the projection maps
maps $p_i: \ul n \to \ul 1$ (sending $s \not= i$ to 0 and sending $i$ to 1) induce a fiber homotopy
equivalence $\rho_I: \ul \cE_I \to \prod_{s=1}^n {\ul\cE}_{i_s} \times_{ \prod {\ul\cB}_{i_s} } {\ul\cB}_I$
compatible with the section $\eta$.
\item
For any $\alpha: \bf n \to \bf m$ in $\cF$,  $I = (i_1,\ldots,i_n)\in \ul\cN(\bf n)$
and any $0 \not= t \in \bf m$,
let $I^{-1}(t) \subset \ul\cN(\bf r)$ denote the ordered subset of $I$  consisting of those entries $i_s$ of $I$
such that $\alpha(s) = t$.   Denote by $\beta: \bf n \to \bf r$ the surjective map
sending $i_s$ to 0 if $i_s \notin I^{-1}(t)$ and otherwise preserving the ordering.  Denote by
$\mu: \bf{r} \to \bf 1$ the map sending each $0 \not= i \in  \bf r$ to $1 \in \bf 1$.
Consider the commutative square
\begin{equation}
\begin{xy}*!C\xybox{%
\xymatrix{  \ul\cE_I \ar[r]^-{\alpha} \ar[d]_{\beta} & \ul\cE_J \ar[d]^{p_t} \\
\ul\cE_{I^{-1}(t)} \ar[r]^-{\mu} & \ul\cE_t \ .
 }
}\end{xy}
\end{equation}
We require that $\mu: \ul\cE_i{I^{-1}(t)} \ \to \ul \cE_t$  restricted to fibers over 
$\mu: \ul\cB_{I^{-1}(t)} \to \ul\cB_t$
be homotopy equivalent to the  smash product
$$\prod_{i \in I^{-1}(t)} (X^i)_\ell  \ \to \ (\prod_{i \in I^{-1}(t)} X^i)_\ell \ \to \ 
(X^t)_\ell \ .$$
\end{enumerate}
\end{defn}

\vskip .1in

As in Definition \ref{defn:X-map}, a map $f \to f^\prime$ of  $\bZ/\ell$-completed $X$-fibrations 
is a commutative diagram of $\cF$-spaces with the property that for each $I \in \cN(\bf n)$
the commutative squares
\begin{equation}
\xymatrix{
({\ul\cE})_I  \ar[r]^{\tilde \phi_I}  \ar[d]_{f_I}  &  ({\ul\cE^\prime})_I \ar[d]^{f^\prime_I} \\
({\ul\cB})_I \ar[r]_{\phi_I }& ({\ul\cB^\prime})_I
}
\end{equation}
are homotopy cartesian.  Two $\bZ/\ell$-completed $X$-fibrations 
$f:\ul\cE \to \ul\cB, \ f^\prime:\ul\cE^\prime \to \ul\cB^\prime $  are said to be  
fiber homotopy equivalent if there exists a chain of fiber homotopy equivalences 
$$\ul\cB \ \stackrel{\phi_1}{\leftarrow}  \ \ul\cB_1 \ \stackrel{\phi_2}{\to} \ \ul\cB_2 \  
\stackrel{\phi_3}{ \leftarrow} \quad  \cdots \quad \ul\cB_n \ \stackrel{\phi_{n+1}}{\to} \ \ul\cB^\prime $$
which are covered by a chain of maps of $\bZ/\ell$-completed $X$-fibrations relating \ $f, \ f^\prime$.

\vskip .1in

\begin{construct}
\label{construct:G(X)-ell}
Assume that $X$ is a smash $\ell$-good, pointed, connected simplicial set.
In a manner parallel to the construction of the $\cF$-space $\ul\cB G(X)$ in Construction
\ref{construct:G(X)},  we define $\ul \cB G_\ell(X): \cF \ \to \ (s.sets_*)$ with
$$\ul \cB G_\ell(X)_I \ = \ \prod_{j=1}^n BG_\ell(X^{i_j}) \times 
\prod_{T \subset {\bf n} } B(Iso(X^{i_T}),Iso(X^{i_T})) ,$$
where $T$ runs over pointed subsets of ${\bf n}$  with more than one non-zero element,
$i_T = \sum_{0 \not= j \in T} i_j$, and $Iso(X^{i_T})$ is the function complex of automorphisms
of $X^{i_T}$.   

Using the action of $G_\ell(X^m)$ on  $(X^m)_\ell$ given by Proposition \ref{prop:G-X-ell}(2),
we adapt Construction \ref{construct:X-action} to define the $\cF$-space 
$\ul \cB (G_\ell(X),X_\ell)$ with 
$$(\ul \cB (G_\ell(X),X_\ell)_I \ = \ \prod_{j=1}^n B(G_\ell(X^{2i_j}),(X^{2i_j})_\ell) \times 
\prod_{T \subset {\bf n} } B(Iso(X^{2i_T}),Iso(X^{2i_T})).$$

Projection determines a natural transformation of functors $\cF \ \to (s.sets_*)$ 
\begin{equation}
\label{eqn:pi-G(X)-ell} 
\pi_{X,_\ell}: \ul \cB (G_\ell(X),X_\ell) \ \to \ \ul \cB G_\ell(X).
\end{equation}
which is a $\bZ/\ell$--completed $X$-fibration over $\ul \cB G_\ell(X)$.
\end{construct}

\vskip .1in

\begin{remark}
\label{rem:contrast}
We emphasize that the construction of $\pi_{X_\ell}: \ul \cB (G_\ell(X),X_\ell) \ \to \ \ul \cB G_\ell(X)$
in (\ref{eqn:pi-G(X)-ell}) differs from that of Construction \ref{construct:G(X)} with $X$ replaced by $X_\ell$.
One difference is that  $(X^m)_\ell$ is not the $m$-fold smash 
product of $X_\ell$; another is that $G_\ell(X) \ \not= \ G(X_\ell)$.
\end{remark}

\vskip .1in

In  the following proposition, we utilize the fiber-wise $\bZ/\ell$-completion 
$(\bZ/\ell)^\bu(-)$ (of a pointed fibration) as given in section I.8 of \cite{Bo-Kan}.

\begin{prop}
\label{prop:X-ell-cases}
Let $X$ be a smash $\ell$-good, pointed, connected simplicial set.
\begin{enumerate}
\item
The fiber-wise $\bZ/\ell$-completion $(\bZ/\ell)_\infty^\bu(f)$ of an $X$-fibration 
$f: \ul\cE \to \ul\cB$ is a $\bZ/\ell$-completed $X$-fibration over $\ul\cB$.
\item
If $f: \ul\cE \to \ul\cB$ is a $\bZ/\ell$-completed $X$-fibration over $\ul\cB$, then the
fiber-wise $\bZ/\ell$-completion $(\bZ/\ell)_\infty^\bu(f)$ is also a $\bZ/\ell$-completed $X$-fibration over $\ul\cB$
and the canonical map $f \to (\bZ/\ell)_\infty^\bu(f)$ is a fiber homotopy equivalence. 
\item
If $f: \ul\cE \to \ul\cB$ is an $X$-fibration and if each $\ul\cB_i$ is simply connected and
$\ell$-good, then $(\bZ/\ell)_\infty(f): (\bZ/\ell)_\infty(\ul\cE) \ \to \ (\bZ/\ell)_\infty(\ul\cB)$
is a  $\bZ/\ell$-completed $X$-fibration over $(\bZ/\ell)_\infty(\ul\cB)$.
\item
Consider a commutative square of $\cF$-spaces
\begin{equation}
\label{eqn:pull-fib}
\xymatrix{
{\ul\cE^\prime}  \ar[r]^{\tilde \phi}  \ar[d]_{f^\prime}  &  {\ul\cE} \ar[d]^{f} \\
{\ul\cB^\prime} \ar[r]^\phi & {\ul\cB}
}
\end{equation}
such that
(\ref{eqn:pull-fib}) is cartesian when restricted to each $I \in \cN(\bf n)$.
If $f^\prime$ is a $\bZ/\ell$-completed $X$-fibration, then so is $f$.

Moreover, if $f$  is an $\bZ/\ell$-completed $X$-fibration and if
$\phi: {\ul\cB^\prime} \ \to \  {\ul\cB}$ is a homotopy equivalence, 
then $f^\prime$ is an $\bZ/\ell$-completed $X$-fibration.
\end{enumerate}
\end{prop}

\begin{proof}
The first assertion follows from the functoriality of the fiber-wise $\bZ/\ell$-completion
functor sending conditions of Definition \ref{defn:X-fibration} to conditions 
of Definition \ref{defn:ell-fibration}.  To prove assertion (2), observe that
if the fibers of $f_I$  are homotopic to $(X^I)_\ell$, then the fibers of
$(\bZ/\ell)^\bu(f_I)$ are homotopy equivalent to $((X^I)_\ell)_\ell)$ which is naturally 
homotopy equivalent to $(X^I)_\ell$.

If each $\ul\cB_i$ is $\ell$-good, then a spectral sequence argument for the $\bZ/\ell$
homology of fibrations implies that for each $I \in \ul\cN(\bf n)$ the natural map from 
the fiber of the $\bZ/\ell$-completion of $f_I: \ul \cE_I \to \ul \cE_I $ to the
$\bZ/\ell$-completion of the fiber of $f_I$, 
$$fib((f_I)_\ell) \quad \to \quad  (fib(f_I))_\ell,$$
is a homotopy equivalence.  This enables us to conclude that $(\bZ/\ell)_\infty(f)$ satisfies 
the conditions  of Definition \ref{defn:ell-fibration}, thereby verifying assertion (3).

To prove assertion (4), first assume that $f^\prime$ is an $\bZ/\ell$-completed $X$-fibration.
The assumption that the restriction of (\ref{eqn:pull-fib}) to each $I$
is a homotopy pull-back square implies that $f$ satisfies the
conditions to be an $\bZ/\ell$-completed $X$ fibration.  The converse is similarly verified, 
 provided that $\phi: \ul\cB^\prime \to \ul\cB$ is a homotopy equivalence.
\end{proof}

\vskip .1in

We give the $\bZ/\ell$-completed analogue of the construction of principal $d$-simplices given in
Definition \ref{defn:principal-simplex}.  We  apply
the fiber-wise $\ell$-completion functor $(\bZ/\ell)_\infty^\bu(-)$ to an $\ell$-completed $X$-fibration
$f: \ul\cE \to \ul\cB$ in order to obtain $\bG_\ell(X^I)$-actions on $(\bZ/\ell)_\infty^\bu(f_I)$ using the natural retraction 
$(\bZ/\ell)_\infty^\bu(-) \circ (\bZ/\ell)_\infty^\bu(-) \ \to \ (\bZ/\ell)_\infty^\bu(-)$.   This is
strictly parallel to the use of $Sin(|Sin(|-|)|) \to Sin(|-|)$ in Lemma \ref{lem:action-Sin}.

\begin{defn}
\label{defn:principal-ell-simplex}
Assume that $X$ is a pointed, connected simplicial set which is smash $\ell$-good.
Let $f: \ul\cE \ \to \ \ul\cB$ be an $\bZ/\ell$-completed $X$-fibration 
and denote by $(\bZ/\ell)_\infty^\bu(f): (\bZ/\ell)_\infty^\bu(\ul\cE) \to \ul\cB$ the fiberwise 
$\ell$-completion of $f$.  Consider some $I = (i_,\ldots,i_n) \in \ul\cN({\bf n})$.   We say that a map $\phi: X^I \times \Delta[d]
 \to (\bZ/\ell)_\infty^\bu(\ul\cE)_I$ is a $\bZ/\ell$-completed principal $d$-simplex of $f$ provided that
 $\phi$ satisfies the following conditions.
 \begin{enumerate}
 \item
 The composition $f_I \circ \phi: X^I\times \Delta[d] \to (\bZ/\ell)_\infty^\bu(\ul\cE)_I \to \ul\cB_I$ factors though the
 projection, $X^I\times \Delta[d] \to \Delta[d] \to \ul\cB_I$  and determines a $\bZ/\ell$-equivalence \\
 $X^I \times \Delta[d] \ \to \ (\bZ/\ell)_\infty^\bu(\ul\cE)_I \times_{\ul\cB_I} \Delta[d]$.
 \item
  For every surjective map $\alpha: {\bf n} \to {\bf r}$ in $\cF$,  set $\alpha(I) = (j_1,\ldots,j_r) \in \ul\cN(r), \ j_s = \sum_{\alpha(i) = s} i_s$.
 Then $\phi$ induces $\phi: X^{\alpha(I)} \times \Delta[d] \to Sin(|(\bZ/\ell)_\infty^\bu(\ul\cE_{\alpha(I)}) |)$ fitting in the commutative square
 \begin{equation}
\xymatrix{
X^I \times \Delta[d] \ar[d]_{\alpha} \ar[r]^-\phi &  Sin(|(\bZ/\ell)_\infty^\bu(\ul\cE)_I)|) \ar[d]^{\alpha} \\
X^{\alpha(I)} \times \Delta[d] \ar[r]^-{\phi^\alpha} & Sin(|(\bZ/\ell)_\infty^\bu(\ul\cE)_{\alpha(I)})|
}
\end{equation}
such that $\phi^\alpha$ induces a weak homotopy equivalence 
$$X^{\alpha(I)} \times \Delta[d] \ \to \ Sin(|(\bZ/\ell)_\infty^\bu(\ul\cE_{\alpha(I)}) |) 
\times_{Sin(|(\bZ/\ell)_\infty^\bu(\ul\cB_{\alpha(I)}|) )} \Delta[d].$$
\end{enumerate}
\end{defn}

\vskip .1in

We proceed to construct the  principal $\cF$-space map over $\ul B$ associated to 
an $\bZ/\ell$-completed $X$-fibration $f: \ul\cE  \to \ul\cB$. 

\begin{construct}
\label{construct:principal-ell}
Let $X$ be a pointed, connected simplicial set which is smash $\ell$-good.
Consider a $\bZ/\ell$-completed $X$-fibration 
 $f: \ul\cE \ \to \ \ul\cB$ \ over the $\cF$-space ${\ul \cB}$ over $\ul \cN$. 
 For any $I = (i_1,\ldots,i_n) \in \ul\cN$, we define \ $P_\ell(f_I)$ \
 to be the sub-complex of the 
simplicial mapping complex $\ul{Hom}_*(X^I,(\bZ/\ell)_\infty^\bu(\ul\cE_I)$ consisting of 
 $\bZ/\ell$-completed principal $d$-simplices of $f_I$.

So defined, $P_\ell(f_I)$ admits a right action by $G_\ell(X^I) \ \equiv \
 \prod_{s=1}^n G_\ell(X^{i_s})$.  Namely, if $\alpha: \Delta[d]\times X^I \to X_\ell^I$ is a 
 $d$-simplex of $G_\ell(X^I)$ and $\phi: \Delta[d] \times X^I \to (\bZ/\ell)_\infty^\bu(\ul\cE)_I$ is a principal $d$-simplex,
 then we define $\phi\cdot \alpha$ as the composition
 \begin{equation}
 \label{eqn:ell-action}
 \Delta[d]\times X^I \ \stackrel{1\times\alpha} \to  \Delta[d]\times (X^I)_\ell \ \stackrel{pr_1 \times\beta_\ell}{\to} \
 \Delta[d] \times ((\bZ/\ell)_\infty^\bu(\ul\cE_I))_\ell 
 \ \stackrel{1 \times \eta}{\to} \Delta[d]\times (\bZ/\ell)_\infty^\bu(\ul\cE_I)
 \end{equation}
 where the last map is associated to the canonical map map $\eta: (\bZ/_\ell)_\infty(-) \circ 
 (\bZ/_\ell)_\infty(-)  \to (\bZ/_\ell)_\infty(-)$.

For any $n > 0, \ I\in \cF({\bf n})$ and any $\alpha: {\bf n} \to {\bf m}$, the following square commutes
\begin{equation}
\label{eqn:prin-compat}
\xymatrix{
P(f_I) \times G(X^I) \ar[r] \ar[d]_\alpha &  P(f_I) \ar[d]^\alpha \\
P(f_{\alpha(I})) \times G(X^{\alpha(I)}) \ar[r] & P(f_{\alpha(I)}) .
}
\end{equation}
By construction, $P((\bZ/\ell)_\infty^\bu(f_I))$ is a principal $G_\ell(X^I)$-fibration over $\ul\cB_I$.
\end{construct}

\vskip .1in

We now construct $\cF$-spaces associated with the action of $G_\ell(X^I)$ on $P((\bZ/\ell)^\bu(f_I))$.

\begin{construct} 
\label{construct:principal-ell-spaces}
Retain the notation and hypotheses of Construction \ref{construct:principal-ell}.
There is a commutative diagram of simplical sets
\begin{equation}
\label{eqn:2-sided-ell}
\xymatrix{
 (\bZ/\ell)_\infty^\bu(\ul\cE_I) \ar[d]^{f_I}  & \ar[l]_-{\tilde p_{f,I}}  B(P((\bZ/\ell)_\infty^\bu(f_I)),G_\ell(X^I),X_\ell^I )) 
 \ar[d]^{\pi_{X,I}} \ar[r]^-{\tilde q_{f,I}}  &   
 B(G_\ell(X_\ell^I),X^I )) \ar[d]^{\pi_{X,I}} &  \\
\ul\cB_I   & \ar[l]^-{p_{f,I}}  B(P((\bZ/\ell)_\infty^\bu(f_I)),G_\ell(X^I)) \ar[r]_-{q_{f,I}} &  BG_\ell(X^I)
}
\end{equation}
for each $I \in \ul\cN({\bf n})$, where the right horizontal maps are projections, the lower left map
is given by projection to $P((\bZ/\ell)_\infty^\bu(f_I))$ followed by the structure map
$P((\bZ/\ell)_\infty^\bu(f_I)) \to \ul\cB_I$ ,
and the upper left map is given by
sending a pair of simplices $\phi: \Delta[d] \times X^I \to P((\bZ/\ell)_\infty^\bu(f_I))$ and 
$\alpha: \Delta[d] \times X^I \to X^I_\ell$ to
$\phi \cdot \alpha$ as given in (\ref{eqn:ell-action}).

The vertical maps of (\ref{eqn:2-sided-ell}) are $\bZ/\ell$-completed $X$-fibrations, the
left and right squares of (\ref{eqn:2-sided-ell}) constitute maps of $\bZ/\ell$-completed $X$-fibrations, 
and the left square is a homotopy
equivalence of $\bZ/\ell$-completed $X$-fibrations.
\end{construct}

\vskip .1in

We basically repeat the proof of Proposition \ref{prop:func-P(f)} to verify the following
$\bZ/\ell$-completed analogue.  

\vskip .1in

\begin{prop}
\label{prop:func-P(f)-ell}
Consider a pointed, connected simplicial set $X$ which is smash $\ell$-good.
Let  $f: \ul\cE \ \to \ \ul\cB$ be a $\bZ/\ell$-completed $X$-fibration 
 over the $\cF$-space ${\ul \cB}$ over $\ul \cN$ with section $\eta: \ul\cB \to \ul\cE$. 
The diagrams of (\ref{eqn:2-sided-ell}) for all $I \in \ul\cN({\bf n}), n > 0$ determine
a commutative diagram of $\cF$-spaces
\begin{equation}
\label{eqn:2-sided-func-ell}
\xymatrix{
 \ul\cE \ar[dr]_f \ar[r]^-\epsilon & (\bZ/\ell)_\infty^\bu(\ul\cE)  \ar[d]^{(\bZ/\ell)_\infty^\bu(f)} &
 \ar[l]_-{\tilde p_{f,\ell}} \ul\cB (P(\bZ/\ell)_\infty^\bu(f),G_\ell(X),X_\ell) \ar[d]_{p_{f,X,\ell}} \ar[r]^-{\tilde q_{f,\ell}} &   
 \ul\cB (G_\ell(X),X_\ell) \ar[d]_{\pi_{X,\ell}}   \\
& \ul\cB   & \ar[l]^-{p_{f,\ell}} \ul\cB (P(\bZ/\ell)_\infty^\bu(f),G_\ell(X)) \ar[r]_-{q_{f,\ell}} &  \ul\cB G_\ell(X)
}
\end{equation}
whose vertical maps are $\bZ/\ell$-completed $X$-fibrations and whose left horizontal maps are homotopy
equivalences.  The map $\epsilon: \ul\cE   \to (\bZ/\ell)_\infty^\bu(\ul\cE)$
appearing in the upper left corner of (\ref{eqn:2-sided-func-ell}) is an equivalence, induced by the natural 
transformation $id(-) \to  (\bZ/\ell)_\infty^\bu(-)$.
\end{prop}

The proof of Theorem \ref{thm:X-ell-universal} given below is a straight forward adaption of the proof of
Theorem \ref{thm:X-universal} to the $\bZ/\ell$-completed context.

\begin{thm} 
\label{thm:X-ell-universal}
Let $X$ be a pointed, connected simplicial set which is smash $\ell$-good.
The $\bZ/\ell$-completed $X$-fibration of $\cF$-spaces over $\ul\cN$ given
in (\ref{eqn:pi-G(X)-ell}),
\ $\pi_{X,\ell}: \ul \cB (G_\ell(X),X_\ell) \ \to \ \ul \cB G_\ell(X),$ \
satisfies the following universal property:

For any special $\cF$-space ${\ul \cB}$ over $\ul\cN$ and any $\ell$-completed $X$-fibration 
$f: {\ul \cE} \ \to {\ul\cB}$ over ${\ul\cB}$,
there is a unique homotopy class of maps of $\cF$-spaces $\phi: {\ul \cB} \ \to \ {\ul\cB}G_\ell(X)$
such that $\phi^*(\pi_{X,\ell})$ is fiber homotopy equivalent to $f$ as $\bZ/\ell$-completed 
$X$-fibrations over ${\ul\cB}$.
\end{thm}

\vskip .1in

\begin{remark}
\label{rem:ell-local}
The discussion of this section culminating in Theorem \ref{thm:X-ell-universal} applies
essentially verbatim to the ``Bousfield-Kan localization at $\ell$"  and the ``Bousfield-Kan rational
localization".   Namely, we simply replace $(\bZ/\ell)_\infty(-)$ by $(Z_{(\ell)}) _\infty(-)$ 
and $(\bZ/\ell)_\infty(-)$ by $\bQ_\infty(-)$.  
 \end{remark}
 
 \vskip .2in


\section{Oriented $\bZ/\ell$-completed $X$-fibrations}
\label{sec:or-ell-complete}

In this section, we discuss  modifications of Section \ref{sec:ell-complete} needed 
to replace the monoid $\bG_\ell(X)$ of $\bZ/\ell$-equivalences
$X \to X_\ell$ by its connected component $\bG_\ell^o(X)$ and impose conditions on $X$
such that $\bG_\ell^o(X)$ is ``well-behaved".  For us, the relevant example is that of $X = S^2$.

\vskip .1in

 \begin{defn}
 \label{defn:orient}
 A $\bZ/\ell$-completed $X$-fibration $f: \ul\cE \to \ul\cB$ (as in Definition \ref{defn:ell-fibration}) 
 is said to be an oriented $\bZ/\ell$-completed $X$-fibration if the homotopy equivalences
 on fibers of Definition \ref{defn:ell-fibration} are homotopic to the identity and if the maps 
 on fibers of Definition \ref{defn:ell-fibration}(3) are homotopic to the appropriate smash products.
 
 Two oriented $\bZ/\ell$-completed $X$-fibrations 
$f:\ul\cE \to \ul\cB, \ f^\prime:\ul\cE^\prime \to \ul\cB^\prime $  are said to be 
oriented fiber homotopy equivalent if there exists a chain of homotopy equivalences 
$$\ul\cB \ \stackrel{\phi_1}{\leftarrow}  \ \ul\cB_1 \ \stackrel{\phi_2}{\to} \ \ul\cB_2 \  
\stackrel{\phi_3}{ \leftarrow} \quad  \cdots \quad \ul\cB_n \ \stackrel{\phi_{n+1}}{\to} \ \ul\cB^\prime $$
which are covered by a chain of maps of between oriented $\bZ/\ell$-completed $X$-fibrations relating 
\ $f, \ f^\prime$ whose maps on fibers are homotopic to the identity.
 \end{defn}
 
 \vskip .1in
 
 \begin{construct}
 \label{construct:G(X)-ell-or}
 Assume that $X$ is a pointed, connected simplicial set such that permuting factors of  the smash product 
$X^n$ for each $n > 0$ is weakly homotopic to the identity.
In a manner parallel to the construction of the $X$-fibration $\pi_X$ in Construction
\ref{construct:G(X)}, the oriented $X$-fibration $\pi_{X,o}$ in Remark \ref{rem:orient},
 and the $\bZ/\ell$-completed $X$-fibration $\pi_{X,\ell}$ in Construction
\ref{construct:G(X)-ell},  we define the oriented $\bZ/\ell$-completed $X$-fibration over $\ul \cB G^o_\ell(X)$
\begin{equation}
\label{eqn:pi-G(X)-ell-or} 
\pi_{X,_\ell,o}: \ul \cB (G^o_\ell(X),X) \ \to \ \ul \cB G^o_\ell(X).
\end{equation}

Thus, we set
$$\ul \cB G_\ell^o(X)_I \ = \ \prod_{j=1}^n BG_\ell^o(X^{i_j}) \times 
\prod_{T \subset {\bf n} } B(Iso(X^{i_T}),Iso^o(X^{i_T})),$$
where $T$ runs over pointed subsets of ${\bf n}$  with more than one non-zero element,
$i_T = \sum_{0 \not= j \in T} i_j$, and $Iso^o(X^{i_T})$ is the connected component of
the function complex of automorphisms of $X^{i_T}$.   
Using the action of $G^o_\ell(X^m) \hookrightarrow G_\ell(X^m)$ on $(X^m)_\ell$ 
given by Proposition \ref{prop:G-X-ell}(2), we set 
$$(\ul \cB (G^o_\ell(X),X_\ell)_I \ = \ \prod_{j=1}^n B(G^o_\ell(X^{2i_j}),(X^{2i_j})_\ell) \times 
\prod_{T \subset {\bf n} } B(Iso(X^{2i_T}),Iso(X^{2i_T})).$$
The map $\pi_{X,_\ell,o}$ is the evident projection.
\end{construct}
 
 \vskip .1in
 
 With notational changes, the proof of Theorem \ref{thm:X-ell-universal} applies
 to prove the following oriented version.

\begin{thm} 
\label{thm:X-ell-universal-o}
Let $X$ be a pointed, connected simplicial set which is smash $\ell$-good such
that $\Sigma_n \hookrightarrow G^o(X^n)$ for all $n > 0$.
Then \ $\pi_{X,\ell,o}$ \ in (\ref{eqn:pi-G(X)-ell-or})
satisfies the following universal property:

For any special $\cF$-space ${\ul \cB}$ over $\ul\cN$ and any oriented $\bZ/\ell$-completed $X$-fibration 
$f: {\ul \cE} \ \to {\ul\cB}$ over ${\ul\cB}$,
there is a unique oriented homotopy class of maps of $\cF$-spaces $\phi: {\ul \cB} \ \to \ {\ul\cB}G^o_\ell(X)$
such that $\phi^*(\pi_{X,\ell,o})$ is fiber homotopy equivalent to $f$ as an oriented $\bZ/\ell$-completed 
$X$-fibrations over ${\ul\cB}$.
\end{thm}

\vskip .1in

The following proposition applies in particular to $X$ with $|X| \simeq S^2$.
The assertions are basically given by Proposition VI.7 of \cite{Bo-Kan}.  We remind the reader
that a pointed, connected topological space is said to be nilpotent if 
$\pi_1(T))$ acts nilpotently on $\pi_i(T)$ for any $i \geq 1$

\begin{prop}
\label{prop:fingen}
Let $X$ be a finite connected simplicial set such that $|X^m|$ is a nilpotent space
with finitely generated homotopy groups for some $m >0$.
\begin{enumerate}
\item
The homotopy groups of $G^o(X^m)$ are finitely generated abelian groups.
\item
The map $\pi_*(G^o(X^m)) \ \to \ \pi_*((G^o(X^m)_\ell)$ is given by tensoring with $\bZ_\ell$.
\item
The map $\pi_*(G^o(X^m)) \ \to \ \pi_*(G_\ell^o(X^m))$ is given by tensoring with $\bZ_\ell$.
\item
The natural map $(G^o(X^m))_\ell \ \to \ G^o_\ell(X^m)$ is a homotopy equivalence.
\end{enumerate}

In particular, if $|X| \simeq S^2$, $\pi_0(G^o(X^m)) \simeq \bZ$ has index 2 in 
$\pi_0(\Omega^\infty S^\infty(S^0))$; for $i > 0$ and $m$ sufficiently large, 
$\pi_i(G^o(X^m))$ is the $i$-th stable homotopy of spheres.  See Example \ref{ex:sphere-spectrum}. 
\end{prop}

\begin{proof}
Assertion (1) is presumably very classical, one reference is \cite[Prop 7.2.(i)]{Bo-Kan}.
By \cite[Prop 7.2.(iii)]{Bo-Kan} (with $W = S^n \wedge X^m$, where $S^n$ here is a finite
simplical set with $|S^n|$ the topological $n$-sphere), this implies that $\pi_n(G^o(X^m))
\to \pi_n(G^o_\ell(X^m))$ induces the isomorphism of assertion (2).  

Moreover, \cite[Prop 7.3(ii)]{Bo-Kan} implies assertion (3).

Thus, assertions (2) and (3) imply that the natural map $(G^o(X^m))_\ell \ \to \ G^o_\ell(X^m)$ is a
weak equivalence.  Since this is map of Kan complexes, it is a homotopy equivalence.
\end{proof} 

\vskip .1in

Proposition \ref{prop:fingen} in conjunction with the explicit constructions of
Construction \ref{construct:G(X)-ell-or} immediately imply the following corollary.
  
 \begin{cor}
\label{cor:cover}
Let $X$ be a finite connected simplicial set such that $|X^m|$ is a nilpotent space
with finitely generated homotopy groups for all $m >0$.
The natural map $\cB G^o(X)  \to \cB G^o_\ell(X)$ induces a homotopy equivalence
of $\cF$-spaces over $\ul \cN$
\begin{equation}
\label{eqn:S2-hom-equiv}
\lambda: (\ul \cB G^o(X))_\ell \quad \stackrel{\approx}{\to} \quad  \ul \cB G^o_\ell(X).
\end{equation}
\end{cor}

\vskip .1in

Proposition \ref{prop:fingen} leads to the following identification of the homotopy groups
of $|| \ul\cB (G^o_\ell(X) ||$.

\begin{prop}
\label{prop:stable-homotopy}
Assume that $|X| \simeq S^2$ and consider the $\Omega$-spectrum $|| \ul\cB (G^o_\ell(X) ||$.
Then $\pi_{i+1}(|| \ul\cB (G^o_\ell(X) ||)$ can be naturally identified with the $\ell$-adic completion
of the $i$-th stable homotopy groups of the spheres for $i > 0$, whereas $\pi_1(|| \ul\cB (G^o_\ell(X) ||)$
can be identified with a subgroup of index 2 in $\bZ_\ell$, the $\ell$-adic completion
of the $0$-th stable homotopy groups of spheres.
\end{prop}

\begin{proof}
Because $\ul\cB (G^o_\ell(X)$ is a special $\cF$-space, the 0-space $|| \ul\cB (G^o_\ell(X) ||)_0$
is the homotopy-theoretic group completion of $\ul\cB (G^o_\ell(X)({\bf 1})$.  This implies that \\
$\pi_{i+1}(|| \ul\cB (G^o_\ell(X) ||)_0)$ equals $\varinjlim_m \pi_i(G^o_\ell(X^m))$.  Thus,
the proposition follows directly from the computations of Proposition \ref{prop:fingen}.
\end{proof}

 \vskip .1in

The $\bZ/\ell$-completion of the universal $X$-fibration 
$\pi_X: \ul \cB (G(X),X) \to \ul \cB G(X)$ of Theorem \ref{thm:X-universal} is 
difficult to compare directly with the universal $\bZ/\ell$-completed $X$-fibration
$\pi_{X,\ell}: \ul \cB (G_\ell(X),X_\ell) \to \ul \cB G(X_\ell)$ of 
Theorem \ref{thm:X-ell-universal} even when $|X|$ is a 2-sphere.
Thanks to Corollary \ref{cor:cover}, we do have such a direct comparison
provided we replace $G(X)$ by $G^o(X)$.

\begin{prop}
\label{prop:pi-o}
Let $X$ be a finite connected simplicial set such that $|X^m|$ is a nilpotent space
with finitely generated homotopy groups for all $m >0$.
Further assume that $\Sigma_n \hookrightarrow G^o(X^n)$.
\begin{enumerate}
\item
The result of applying $(\bZ/\ell)_\infty(-)$ to $\pi_{X,o}$,
\begin{equation}
\label{eqn:pi-prime}
(\pi_{X,o})_\ell: ({\ul\cB}(G^o(X),X))_\ell  \ \to \ ({\ul\cB}G^o(X))_\ell \ ,
\end{equation}
is an oriented $\bZ/\ell$-completed $X$-fibration.
\item
The commutative square of $\cF$-spaces
\begin{equation}
\label{eqn:comm-cF}
\xymatrix{
(\ul \cB (G^o(X),X))_\ell \ar[r]  \ar[d]^{(\pi_{X,o})_\ell} &  \ul \cB (G^o_\ell(X),X_\ell) 
\ar[d]^{\pi_{X,\ell,o}} \\
(\ul \cB G^o(X))_\ell \ar[r]_{\lambda} & \ul \cB G^o_\ell(X)
}
\end{equation}
is a fiber homotopy equivalence $(\pi_{X,o})_\ell \ \stackrel{\approx}{\to}\ \pi_{X,\ell,o} $
of oriented $\bZ/\ell$-completed $X$-fibrations.  
\end{enumerate}
\end{prop}

\begin{proof}
The assertion that (\ref{eqn:pi-prime}) is a $\bZ/\ell$-completed $X$-fibration
follows from Proposition \ref{prop:X-ell-cases}(3), since each $(\ul \cB G^o(X))_i = BG^o(X^i)$
is simply connected and each $X^i$ is $\ell$-good (see (\cite[Prop 5.3]{Bo-Kan}).  
Since $\pi_{X,o}$ is oriented, one readily concludes that 
$(\pi_{X,o})_\ell$ is also oriented using the fact that the functor $(\bZ_\ell)_\infty(-)$
sending $G(X^i)$ to $G_\ell(X^i)$ restricts to $G^o(X^i) \to G^o_\ell(X^i)$.

The naturality of $ G^o(X)_\ell \to G^o_\ell(X_\ell)$ with respect to $X$ implies the commutativity
of the  square (\ref{eqn:comm-cF}).
One readily checks that the commutative squares indexed by $I \in \cN(\bf n)$ constituting 
(\ref{eqn:comm-cF}) are homotopy cartesian and preserve orientations, thereby implying the second assertion.
\end{proof}

\vskip .1in
\begin{cor}
\label{cor:classify}
Let $X$ be as in Proposition \ref{prop:pi-o} and assume that the $\cF$-space
$\ul\cB$ satisfies the condition that $(\ul\cB)_I$ is simply connected and $\ell$-good
for all $I \in \ul\cN(\bf n)$ for all $n > 0$. 
Let $f: \ul\cE \ \to \ \ul\cB$ be an oriented $X$-fibration with classifying map
$\phi: \ul\cB  \ \to \ \ul\cB G^o(X)$.
Then $(\bZ/\ell)_\infty(f): \ul\cE_\ell \ \to \ \ul\cB_\ell$
is an oriented $\bZ/\ell$-completed $X$-fibration with classifying map
$$\lambda \circ (\bZ/\ell)_\infty(\phi):  \ul\cB_\ell \ \to \ (\ul\cB G^o(X))_\ell \ \stackrel{\approx}{\to} \ \ul\cB G^o_\ell(X).$$
\end{cor}

\begin{proof}
The proof of Proposition \ref{prop:pi-o}(1) applies to verify that 
$(\bZ/\ell)_\infty(f): \ul\cE_\ell \ \to \ \ul\cB_\ell$ is an oriented $\bZ/\ell$-completed $X$-fibration.
The fact that the classifying map of this oriented $\bZ/\ell$-completed $X$-fibration is given by
$\lambda \circ (\bZ/\ell)_\infty(\phi)$ follows from the commutativity of the following diagram
of $\cF$-spaces
\begin{equation}
\label{eqn:comm-cF-2}
\xymatrix{
 \ul\cE_\ell \ar[d]_{f_\ell} \ar[r] & \cB (G^o(X),X))_\ell \ar[r]  \ar[d]^{(\pi_{X,o})_\ell} &  \ul \cB (G^o_\ell(X),X_\ell) 
\ar[d]^{\pi_{X,\ell,o}} \\
\ul\cB_\ell \ar[r]_{\phi_\ell} & (\ul \cB G^o(X))_\ell \ar[r]_{{\lambda}} & \ul \cB G^o_\ell(X) \ .
}
\end{equation}
\end{proof}

\vskip .2in


\section{Variants of the $J$-homomorphism}
\label{sec:J-hom}

\vskip .1in

The classical $J$-homomorphism relates the 0-connected spectrum ${\bf kU}$ of 
complex $K$-theory  to the (first 
delooping) of the sphere spectrum.  Proposition  \ref{prop:J} represents the $J$-homomorphism
as the classifying map $\cJ$ for the (oriented) $S^2$-fibration  $\tau_{S^2}$ in (\ref{eqn:tau}). 

Most of this section discusses  how to relate the $\bZ/\ell$-completion of $\tau_{S^2}$ to a $\bZ/\ell$-completed 
$S^2$-fibration arising in the context of simplicial schemes over $k$.  This will be used in conjunction 
with the relation of $(\psi^p)_\ell$ to the Frobenius map for schemes over $k$ discussed 
in Section \ref{sec:Adams}.

We shall find it convenient to refer to any $X$-fibration over $\ul \cB$
as an $S^2$-fibration if $|X|$ is homotopy equivalent to the 2-sphere.

\vskip .1in
We begin by recalling the construction of the $\cF$-space ${\ul \cB}GL(\bC)$ 
determining the spectrum ${\bf kU}$.

\begin{construct} \cite[Ex 8.1]{F80}
\label{construct:tau-S2}
There exists an $\cF$-object over $\ul\cN$ of pointed topological spaces 
$${\ul \cB}GL(\bC)^{top}: \cF \quad \to \quad (\text{spaces}_*)$$
with the property that ${\ul \cB}GL(\bC)^{top}({\bf 1})\ = \ \coprod_{n\geq 0} BGL_n(\bC)^{top}$, \
where $BGL_n(\bC)$ is the total topological space associated to the simplicial bar construction 
applied to the complex Lie group $GL_n(\bC)$.  Let $(\bC^n)^+$ denote the 
one point compactification of $\bC^n$, leading to a similarly constructed $\cF$-object
${\ul \cB}(GL(\bC),\bC^+)^{top}$ with
${\ul \cB}(GL(\bC),\bC^+)^{top}({\bf 1})$ \ equal to  \ $\coprod_{n\geq 0} B(GL_n(\bC),(\bC^n)^+)^{top}$. 

Then the result of applying $Sin(-)$ to the projection  ${\ul \cB}(GL(\bC),\bC^+)^{top} \ \to \ {\ul \cB}GL(\bC)^{top}$ 
is an $S^2$-fibration in the sense of Definition \ref{defn:X-fibration}; we denote this $S^2$-fibration by
\begin{equation}
\label{eqn:tau} 
\tau_{S^2}: {\ul \cB}(GL(\bC),\bC^+)  \quad \to \quad  {\ul \cB}GL(\bC).
\end{equation}
\end{construct}

\vskip .1in

The following result is due to G. Segal \cite{Segal}.

\begin{prop}
The ring spectrum $||\ul\cB GL(\bC)||$ associated to the $\cF$-space of Construction \ref{construct:tau-S2}
is equivalent to $\bf{kU}$.
\end{prop}

\vskip .1in

The proof of the following proposition is evident from its statement.

\begin{prop}
\label{prop:J}
The action of $GL_n(\bC)$ on $\bC^n$ determines a natural embedding
$Sin(GL_n(\bC)) \hookrightarrow G^o(X^n)$ where $X$ is a pointed simplicial set 
with $|X| \ \simeq \ \bC^+$.  Thus, the classifying map for
$S^2$-fibration $\tau_{S^2}$ has the form
$$\cJ \ \equiv \ \lambda \circ \cJ^o: {\ul \cB}GL(\bC) \ \to \ {\ul \cB} G^o(X) \ \to \ {\ul \cB} G(X).$$

We define
the $J$ homomorphism to be the map of spectra determined by $\cJ$,
\begin{equation}
\label{eqn:bJ}
{\bf J}: ||{\ul\cB} GL(\bC)|| \quad \to \quad ||{\ul\cB} G(X)||.
\end{equation}
\end{prop}

\vskip .1in

The technique we employ involves relating the $J$-homomorphism to constructions in 
algebraic geometry, following the work of Quillen \cite{Quillen68}, \cite{Quillen} and Sullivan \cite{Sul}
proving the original Adam's conjecture as well as the author's own proof \cite{F73}.

We begin by relating $\tau_{S^2}$  to its algebraized analogue $\tau_{|S^2|}$.  
Our argument involves various steps in order to relate the
topological action of $GL_n(\bC)$ on $(\bC^n)^+$ to the algebraic action of $GL_{n,\bC}$
on $|S^{2n}_\bC|$ as considered in Proposition \ref{prop:map-cone}.
In doing so, we require the following $S^2$-fibrations which are constructed in a 
parallel manner to that given in Construction \ref{construct:tau-S2}.

\begin{prop}
\label{prop:further-actions}
The construction of  $\tau_{S^2}$ of (\ref{eqn:tau}) has the following variants.
\begin{enumerate}
\item
$\tau^U_{S^2}: \ul\cB (\bU,S^2) \ \to \ \ul\cB \bU$ 
associated to the action of the unitary subgroup $\bU_n \subset GL_n(\bC)$
on $S^2 \hookrightarrow (\bC^n)^+$
\item
$\tau_{D/S}: \ul\cB(\bU,D^{2n}/S^{2n-1}) \ \to \ \ul\cB \bU$ 
associated to the action of $\bU_n$ on the 
quotient of the unit disk
$D^{2n}$ by the unit sphere $S^{2n-1}$ in $\bR^{2n} \simeq \bC^n \subset (\bC^n)^+$.
\item
$\tau^U_{|S^2|}: \ul\cB(\bU,|S^2_C|) \ \to \ \ul\cB \bU$ 
associated to the action of $\bU_n$ on $|S^2_\bC|$ (defined in Proposition \ref{prop:map-cone}).
\item
$\tau_{|S^2|}: {\ul \cB}(GL(\bC),|S^2_\bC|) \ \to \ {\ul \cB}GL(\bC)$
associated to the action of $GL_n(\bC)$ on $|S^2_\bC|$.
\end{enumerate}
\end{prop}

 \vskip .1in
  
 The  chain of homotopy equivalences of $S^2$-fibrations occurring in 
 Proposition \ref{prop:further-actions} 
 circumvents the lack of $GL_n(\bC)$-equivariant homotopy equivalences
 between $S^{2n}$ and $|S_\bC^{2n,top}|$  respecting smash products.
 We leave the proof  to the reader.
 
 \begin{prop}
 \label{prop:J-compat}
 The following commutative diagram of $\cF$-spaces  determined by (\ref{eqn:cone-maps}) 
 and the construction of  $\tau_{S^2}$ of (\ref{eqn:tau})
\begin{equation}
\label{eqn:GL-alg-top}
\xymatrix{
{\ul\cB}(GL(\bC),\bC^+) \ar[d]^{\tau_{S^2}} & \ar[l] {\ul\cB}(\bU,S^2) \ar[d]^{\tau_{S^2}^U} \ar[r]
& {\ul\cB}(\bU,D^2/S^1) \ar[d]^{\tau_{D/S}} \\
{\ul\cB}GL(\bC) & \ar[l ]  \ar[r] {\ul\cB}\bU & 
{\ul\cB}\bU 
}
\end{equation}

\begin{equation}
\xymatrix{
{\ul\cB}(\bU,D^2/S^1) \ar[d]^{\tau_{D/S}}  & \ar[l] {\ul\cB}(\bU,|S^2_\bC|) 
\ar[d]^{\tau_{|S^2|}^U} \ar[r]
&   {\ul\cB}(GL,|S^2_\bC|) \ar[d]^{\tau_{|S^2|}}  \\
{\ul\cB}\bU  & \ar[l] {\ul\cB}\bU \ar[r] &  {\ul\cB} GL(\bC) \ .
}
\end{equation}
consists of  a chain of oriented homotopy equivalences of $X$-fibrations, thereby determining 
a homotopy equivalence $\tau_{S^2} \ \approx \ \tau_{|S^2|}$
of oriented $X$-fibrations.
\end{prop}

\vskip .1in

We supplement Proposition \ref{prop:J-compat} with the following fiber homotopy equivalence
of $\ell$-completed $X$-fibrations with $|X|$ homotopy equivalent to $S^2$.

\begin{prop}
\label{prop:class-compare}
The following homotopy commutative diagram of $\cF$-spaces
\begin{equation}
\label{eqn:C-to-Cwedge}
\xymatrix{
( {\ul \cB} (GL(\bC),(\bC)^+))_\ell \ar[d]^{(\tau_{S^2})_\ell} \ar[r] &
({\ul \cB} (GL(\bC),|S^{2,alg}_\bC|))_\ell \ar[d]^{(\tau_{|S^2|})_\ell} \ar[r]^{\tilde \gamma} &
({\ul\cB} (GL_{\bC},S_\bC^{2,alg}))^\wedge \ar[d]^{(\tau_\bC)^\wedge}    
 \\
 ({\ul \cB} GL(\bC))_\ell  \ar[r]_= &
({\ul \cB} GL(\bC))_\ell  \ar[r]_\gamma &
({\ul\cB} GL_{\bC})^\wedge \ ,
}
\end{equation}
determines fiber homotopy equivalences $(\tau_{S^2})_\ell \ \to \ (\tau_{|S^2|})_\ell \ \to \ (\tau_\bC)^\wedge$
of oriented $\bZ/\ell$-completed $X$-fibrations where $|X| \simeq S^2$.  Here,
\begin{equation}
\label{eqn:gamma}
\gamma: ({\ul \cB} GL(\bC))_\ell  \ \to \ ({\ul\cB} GL_{\bC})^\wedge
\end{equation}
denotes a representative of the homotopy class of homotopy equivalences
of $\cF$-spaces associated to the ``classical comparison theorem" as in (\ref{eqn:Phi_X}).
\end{prop}

\begin{proof}
The upper left map of (\ref{eqn:C-to-Cwedge}) is justified using Proposition \ref{prop:alg-smash}
and functoriality of $(\bZ/\ell)_\infty(-)$.  The upper and  lower right maps are the homotopy equivalences
of $\cF$-spaces associated to the ``classical comparison theorem" as in (\ref{eqn:Phi_X}). 
By construction, the squares of (\ref{eqn:C-to-Cwedge}) are commutative and homotopy cartesian,
thus fiber homotopy equivalences of oriented $\ell$-completed $X$-fibrations with $|X| \simeq S^2$.
\end{proof}

\vskip .1in

Applying Theorem \ref{thm:X-ell-universal-o}, we conclude that Proposition \ref{prop:class-compare}
has the following corollary.

\begin{cor}
\label{cor:J-bC}
Denote by 
$$(\cJ_\bC^o)^\wedge: ({\ul\cB} GL_{\bC})^\wedge  \quad \ \to \ul\cB G_\ell^o(S^2)$$
the classifying map for $(\tau_\bC)^\wedge: ({\ul\cB} (GL_{\bC}^{alg},S_\bC^{2,alg}))^\wedge \to 
({\ul\cB} GL_{\bC})^\wedge $.   Then 
$$\cJ_\ell^o, \ (\cJ_\bC^o)^\wedge \circ \gamma: ({\ul\cB} GL(C))_\ell  \quad \ \to \quad \cB G_\ell^o(S^2)$$
are homotopy equivalent maps of $\cF$-spaces.
\end{cor}

As in Proposition \ref{prop:holim}, functoriality of $(-)^\wedge: (\text{pointed simplicial schemes}) 
\to (s.sets_*)$ enables us to associate an $\cF$-space to a functor $\cF \to (\text{pointed simplicial schemes})$.  

\vskip.1in

\begin{prop}
\label{prop:tau-hat}
Let $R$ be either $\bC$ or $W(k)$ or $k$.
There is a natural  map of $\cF$-objects of pointed simplicial $R$-schemes
\begin{equation}
\label{eqn:tau-alg}
\tau_R: {\ul \cB}(GL_R,S_R^{2,alg}) \quad \to \quad {\ul \cB}GL_R \ .
\end{equation}
The evaluation at ${\bf 1} \in \cF$ of $\tau_R$ is the natural projection 
$$ \coprod_{n\geq 0} B(GL_{R,n},S_R^{2n,alg}) \quad \to \quad \coprod_{n\geq 0} BGL_{R,n} \ .$$

Base changes along $\Spec \bC \to \Spec W(k), \ \Spec k \to \Spec W$ determine a commutative
diagram of simplicial schemes
\begin{equation}
\label{eqn:base-change-s}
\xymatrix{
{\ul\cB}(GL_\bC,S^{2,alg}_\bC) \ar[d]^{\tau_{\bC}} \ar[r] & {\ul\cB}(GL_{W(k)},S^{2,alg}_{W(k)})  
\ar[d]^{\tau_{W(k)}}  & \ar[l] {\ul\cB}(GL_k,S^{2,alg}_k) \ar[d]^{\tau_k} \\
\ul\cB GL_\bC \ar[r] & \ul\cB GL_{W(k)} &  \ar[l] \ul\cB GL_k \ .
}
\end{equation}

The result of applying $(-)^\wedge$ to this diagram determines fiber homotopy equivalences
 of oriented $\bZ/\ell$-completed $S^2$-fibrations
\begin{equation}
\label{eqn:base-change-ss}
\xymatrix{
(\ul\cB (GL_\bC,S^{2,alg}_\bC))^\wedge \ar[d]^{(\tau_{\bC})^\wedge} \ar[r] & (\ul\cB (GL_{W(k)},S^{2,alg}_{W(k)}))^\wedge  
\ar[d]^{(\tau_{W(k)}^\wedge)}  & \ar[l] (\ul\cB (GL_k,S^{2,alg}_k))^\wedge \ar[d]^{(\tau_k)^\wedge} \\
(\ul\cB GL_\bC)^\wedge \ar[r] & (\ul\cB GL_{W(k)})^\wedge &  \ar[l] (\ul\cB GL_k)^\wedge \ .
}
\end{equation}

We denote by $\rho: (\ul\cB GL_\bC)^\wedge  \ \to \ (\ul\cB GL_k)^\wedge$ a representative of the homotopy class
of maps of $\cF$-spaces defined by the bottom row of (\ref{eqn:base-change-ss}).
\end{prop}

\begin{proof}
The fact that the $\cF$-spaces $({\ul \cB}GL_R)^\wedge$ and $({\ul \cB}(GL_R,S_R^{alg}))^\wedge$ 
are special $\cF$-spaces follows from Corollary \ref{cor:product}.  The fact that the horizontal maps of
(\ref{eqn:base-change-ss}) are homotopy equivalences follows from Proposition \ref{prop:comparison}.

A functorial isomorphism is given in \cite[Thm 10.7]{F82} relating the mod-$\ell$ \'etale cohomology 
of the geometric fiber of  a map $X \to Y$ of simplicial $R$-schemes such as 
those occurring in Corollary \ref{cor:product} and the mod-$\ell$ cohomology 
of the homotopy-theoretic fiber of $X^\wedge \to Y^\wedge$.
This, together with the fact that each $(BGL_{n,R})^\wedge$ is simply connected, provides the 
functorial homotopy comparisons of the fibers of $(\ul\cB(GL_R,S_R^{2,alg}))^\wedge_I \to 
(\ul\cB GL_R)^\wedge_I$ with $\prod_{j= 1}^n (S^{2i_j})_\ell$ for any $I = (i_1,\ldots,i_n) \in  \ul\cN({\bf n})$
required for $(\tau_R)^\wedge$ to be a $\bZ/\ell$-completed $S^2$-fibration.  

Since this comparison
is compatible with base change for $\Spec \bC \to \Spec R \leftarrow \Spec k$, we conclude that
(\ref{eqn:base-change-ss}) determines a fiber homotopy equivalence of $(\bZ/\ell)$-completed
$S^2$-fibrations covering $\rho$.
\end{proof}

\vskip .1in

\begin{defn}
\label{defn:J-k}
Denote by $(\cJ_k^o)^\wedge: (\ul\cB GL_k)^\wedge \to \ul\cB G^o_\ell(S^2)$ 
 the homotopy class of maps classifying the oriented, $\ell$-completed $S^2$ fibration $(\tau_k)^\wedge$.
Thus, $(\cJ_k^o)^\wedge$  fits in a homotopy commutative map of $\cF$-spaces
\begin{equation}
\label{eqn:j-k}
\xymatrix{
(\ul\cB (GL_k,S^{2,alg}_k))^\wedge \ar[d]^{(\tau_k)^\wedge} \ar[r] & \ \ul\cB (G_\ell^o(S^2),S^2_\ell) \ar[d]^{\pi_{S^2,\ell.o}} \\
(\ul\cB GL_k)^\wedge \ar[r]^{(\cJ_k^o)^\wedge} & \ul\cB G_\ell^o(S^2) \ .
}
\end{equation}
which is a map of $\bZ/\ell$-completed $S^2$-fibrations (i.e., is homotopy cartesian).
\end{defn}

\begin{remark}
\label{rem:base-point}
The \'etale homotopy type of $\Spec R$ with $R$ equal to either $k$ or $W(k)$ or $\bC$ is a point.  Thus,
if $U_{W(k)} \to \Spec W(k)$ is a scheme over $\Spec W(k)$ equipped with a distinguished section $\Spec W(k) \to U_{W(k)}$,
then the base change maps $(U_\bC)_{et} \to (U_{W(k)})_{et} \leftarrow (U_k)_{et}$ have different pointings,
these pointings are related by the ``path" given by the section $\Spec W(k) \to U_{W(k)}$.
\end{remark}

\vskip .1in

The preceding discussion leads to the following proposition.

\begin{prop}
The following diagram of special $\cF$-spaces is homotopy commutative:
\begin{equation}
\label{eqn:hom-comm}
\xymatrix{
( {\ul \cB} GL(\bC))_\ell \ar[r]^\gamma \ar[rd]_{(\cJ^o)_\ell} & (\ul\cB GL_\bC)^\wedge \ar[d]^{(\cJ_\bC^o)^\wedge} \ar[r]^\rho &
(\ul\cB GL_k)^\wedge \ar[ld]^{(\cJ_k^o)^\wedge} \\
& \ul\cB G^o_\ell(S^2) & \quad .
}
\end{equation}
where $\gamma$ is given in (\ref{eqn:gamma}) and  $\rho$ is given by the bottom row 
of (\ref{eqn:base-change-ss}); the homotopy class of the map
$\cJ^o$ is given in Proposition \ref{prop:J}, $(\cJ_\bC^o)^\wedge$ is given in Corollary \ref{cor:J-bC}, 
$(\cJ_k^o)^\wedge$ is given in Definition \ref{defn:J-k}.   
\end{prop}

\begin{proof}
The homotopy commutativity of the left triangle follows from Proposition \ref{prop:class-compare} (see \ref{eqn:C-to-Cwedge})
in conjunction with Theorem \ref{thm:X-ell-universal-o}.  

The homotopy commutativity of the right triangle follows from Proposition \ref{prop:tau-hat} (see \ref{eqn:base-change-s})
in conjunction with Theorem \ref{thm:X-ell-universal-o}.  
\end{proof}

\vskip .2in


\section{Adams operations}
\label{sec:Adams}

The purpose of this section is to present our $\cF$-space model for the Adams operation 
$\psi^p: \bf {ku}_\ell \to \bf {ku}_\ell$ on the $\bZ/\ell$-completed 0-connected spectrum of
complex $K$-theory.  We incorporate the perspective of D. Sullivan  \cite{Sul} who utilizes
the \'etale homotopy type of actions of elements of $Gal(\bC)$ on complex algebraic varieties.

\vskip .1in
We first record the following observation using the last statement of Proposition \ref{prop:Frob}(3).

\begin{lemma}
\label{lem:inverse}
The two maps
$$F^\wedge, \ (\sigma_k)^\wedge: (\ul\cB GL_k)^\wedge \ \to \ (\ul\cB GL_k)^\wedge$$
are mutually inverse to each other.  
\end{lemma}

The following theorem is a restatement of the above lemma in terms of $\cF$-spaces using Sullivan's action
of $\sigma_\bC^{-1} \in Gal(\bC)$ and incorporating Quillen's interpretation of $\psi^p$ for
vector bundles on varieties over $\Spec  k$.

\vskip .1in

\begin{thm} (\cite{Sul}, \cite{Quillen68})
\label{thm:Sullivan}
Let $\sigma_{\bC} \in Gal(\bC/\bQ)$ be an extension to an automorphism of $\bC$ of the Teichm\"uller lifting
$\sigma_{W(k)}: W(k) \to W(k)$ of the arithmetic Frobenius $\sigma \in Gal(k/\bF_p)$.  Then
the following is a homotopy commutative square of ring spectra
\begin{equation}
\label{eqn:sigma} 
\xymatrix{
||(\ul\cB GL(\bC))_\ell || \ar[d]_{(\psi^p)_\ell} \ar[r]^-{||\gamma||} & ||(\ul\cB GL_{\bC})^\wedge|| \ar[d]^{||(\sigma_\bC^{-1})^\wedge||} 
\ar[r]^{||\rho||} &  ||({\ul\cB}GL_k)^\wedge|| \ar[d]^{||F^\wedge||} \\
||(\ul\cB GL(\bC))_\ell || \ar[r]_-{||\gamma||} & ||({\ul\cB} GL_{\bC})^\wedge|| \ar[r]_{||\rho||} &  ||({\ul\cB}GL_k)^\wedge|| \ ,
}
\end{equation}
where the map of $\cF$-spaces $\gamma: (\ul\cB GL(\bC))_\ell \to (\ul\cB GL_{\bC})^\wedge$ is
presented in Corollary \ref{cor:J-bC} (determined by the ``classical comparison theorem" as in 
Proposition \ref{prop:holim}) and the map $\rho:(\ul\cB GL_{\bC})^\wedge \to {\ul\cB}GL_k)^\wedge$
appears in Proposition \ref{prop:tau-hat} (determined by ``base change from $\bC$ to $k$").
\end{thm}

\vskip .1in

\noindent
{\bf Discussion of proof of Theorem \ref{thm:Sullivan}.}

The fact that the maps of (\ref{eqn:sigma}) respect the mutliplicative structure 
of the spectra occurring in the diagram can be verified using Segal's incorporation
of a multiplication on an $\cF$-space.  (See \cite[Defn 5.1]{Segal}.)

By Lemma \ref{lem:inverse}, the right square of (\ref{eqn:sigma}) commutes 
and the horizontal maps $|| \rho ||$ are homotopy equivalences as seen in 
Proposition \ref{prop:tau-hat}.  Thus, the homotopy commutativity of the
outer square is equivalent to the homotopy commutativity of the left square 
(\ref{eqn:sigma}).

Sullivan's proof of the homotopy commutativity of the left square is based on
the splitting principle for bundles over Grassmannians, reducing the verification
to restricting to the two compositions of the left square to $||\ul\cB \bT(\bC)_\ell ||$,
where $\bT_n(\bC) \hookrightarrow GL_n(\bC)$ is the maximal torus of diagonal
invertible $n\times n$ matrices.

Quillen essentially considers the outer square.  Let $E$ be a rank $n$ 
vector bundle over a variety $Y$ over $\Spec k$ defined over $\bF_p$. 
Quillen observes 
that $\psi^p(E)$  is given by the pull-back of $F^*(E)$ of $E$
along the Frobenius map $F: Y \to Y$.  More generally, he observes that
Adams operations can be expressed in terms of $\lambda$-operations
which can be expressed upon addition of a trivial bundle by maps  
$\Lambda^i: Gr_{N+n,n;\bZ} \to Gr_{M+m,m;\bZ}$.  Furtheremore, 
base change maps for $R$ with $R$ = $\bC$ or $W(k)$ or $k$ 
give $\lambda$ operations on Grassmannians over $\Spec R$.

A later proof of Theorem \ref{thm:Sullivan} is given by the author 
in \cite[Prop 2.11]{F76} (with quite different notation) at the 0-space level of 
the spectra in (\ref{eqn:sigma}).  Since  $F^\wedge: ({\ul\cB}GL_k)^\wedge \to ({\ul\cB}GL_k)^\wedge$ respects
the ring structure (associated to tensor product) on these 0-spaces, one should then 
appeal to \cite[VII.3.2]{May77} to conclude that the outer square of (\ref{eqn:sigma}) is a homotopy commutative 
square of ring spectra.

\vskip .2in


\section{The Stable Adams Conjecture}
\label{sec:stable}

The following proposition presents a first version of the stable Adams conjecture.

\vskip .1in
 
 \begin{prop}
 \label{prop:F-wedge}
 Retain the notation of Section \ref{sec:J-hom}.
 \begin{enumerate}
 \item
 The commutative square of $\cF$-spaces
  \begin{equation}
\label{eqn:Fwedge}
\xymatrix{
(\ul\cB (GL_k,S_k^{2,alg}))^\wedge \ar[r]^{F^\wedge} \ar[d]^{(\tau_k^{alg})^\wedge} &  
(\ul\cB (GL_k,S_k^{2,alg}))^\wedge \ar[d]^{(\tau_k^{alg})^\wedge} \\
(\ul\cB GL_k)^\wedge \ar[r]^{F^\wedge}  & (\ul\cB GL_k)^\wedge
}
\end{equation}
obtained by applying $(-)^\wedge$ to (\ref{eqn:S-ell})
is a fiber homotopy equivalence of $\bZ/\ell$-completed $S^2$-fibrations.
\item
Consequently, the following diagram of $\cF$-spaces is homotopy commutative:
\begin{equation}
\label{eqn:F-J}
\xymatrix{
(\ul\cB GL(\bC))_\ell \ar[rrrd]^{\cJ_\ell} \ar[r]^\gamma & (\ul\cB GL_{\bC})^\wedge 
\ar[r]^\rho \ar[dd]_{(\sigma_\bC^{-1})^\wedge} & (\ul\cB GL_k)^\wedge \ar[rd]^{(\cJ_k)^\wedge} \ar[dd]_{F^\wedge} & \\
& & & \ul\cB G_\ell(S^2) \ .\\
(\ul\cB GL(\bC))_\ell \ar[rrru]_{\cJ_\ell}  \ar[r]_\gamma & (\ul\cB GL_{\bC})^\wedge 
\ar[r]_\rho & (\ul\cB GL_k)^\wedge \ar[ru]_{(\cJ_k)^\wedge} &
}
\end{equation}
\end{enumerate}
 \end{prop}
 
 \begin{proof}
 The first assertion follows directly from Proposition \ref{prop:hom-fiber}
and the definitions of the objects/maps considered.

The second assertion follows from the first together with Theorem \ref{thm:X-ell-universal}
and Lemma \ref{lem:inverse}.
 \end{proof}

\vskip .1in

In our proof of Theorem \ref{thm:stable-adams}, we  use  the following lemma ``commuting"
the functors $(\bZ_\ell)_\infty(-)$ with $||-||$.

\begin{lemma}
\label{lem:ell-spectra}
Let $\ul \cB$ be a special $\cF$-space with the property that $\ul\cB_I$ is $\ell$-good for
every $I \in \ul\cN(\bf n), \ n > 0$.   Then the natural map of $\cF$-spaces
\ $\ul\cB  \  \to \ (\ul \cB)_\ell$ \ induces a homotopy equivalence of spectra
\begin{equation}
\label{eqn:specequiv}
(|| \ul\cB ||)_\ell \times_{{\bf K(\bZ_\ell,0)}} {\bf K(\bZ,0)} \quad \to \quad || (\ul \cB)_\ell ||.
\end{equation}

In particular, $||\cJ_\ell ||$ is a model for the $\bZ/\ell$-completion of the 0-connected spectrum ${\bf kU }$
restricted along ${\bf K(\bZ,0)} \to {\bf K(\bZ_\ell,0)}$.
\end{lemma}

\begin{proof}
Since $(\bZ/\ell)_\infty(-)$ does not localize the set of connected components, 
we apply $\times_{{\bf K(\bZ_\ell,0)}} {\bf K(\bZ,0)}$
to the $\bZ/\ell$-completed spectrum $(\bZ/\ell)_\infty(||\ul\cB||)$ in order that this spectrum
be directly related to the spectrum $||(\ul \cB)_\ell ||$. 

Observe that the prolongation of $\ul\cB \to (\ul\cB)_\ell$ 
to a natural transformation of functors on finite pointed simplicial sets (with values in pointed simplicial sets,
as in \cite{Bo-F})
yields an $\bZ/\ell$-equivalence \ $\ul\cB(S)  \  \to \ (\ul \cB)_\ell(S)$  for 
any finite pointed simplicial set $S$.    This implies that $\ul\cB(S^n)  \  \to \ ((\ul \cB)_\ell(S^n), \ n > 0,$
is an $\bZ/\ell$-equivalence.  This last map can be identified as the map on the spaces of 
$n$-th deloopings $||\ul\cB||_n \  \to \ ||(\ul \cB)_\ell ||_n$ of the homotopy theoretic group 
completion of $\ul\cB(\bf 1) \to (\ul\cB_\ell)(\bf 1))$.  (See \cite{Bo-F}.)
\end{proof}

\vskip .1in

We now verify that the translation of Proposition \ref{prop:F-wedge} into the context of 
spectra together Theorem \ref{thm:Sullivan} imply the following theorem (also stated in the
introduction).  

\vskip .1in

\begin{thm}
\label{thm:stable-adams}

Let ${\bf kU}$ denote the 0-connected spectrum of (topological) complex K-theory
and let ${\bf BS^{2}_\ell} \ = \ ||\ul\cB G_\ell(S^2)||$ denote the 0-connected spectrum naturally constructed using
self-equivalences of $\bZ/\ell$-completions of even spheres.

If $p$ and  $\ell$ are distinct primes, then the maps of spectra
$${\bf J}_\ell, \ {\bf J}_\ell \circ (\psi^p)_\ell:  ({\bf kU})_\ell \times_{{\bf K(\bZ_\ell,0)}} {\bf K(\bZ,0)} 
\quad \to \quad {\bf BS^2_\ell} $$
are homotopy equivalent.
\end{thm} 

\begin{proof}
By Lemma \ref{lem:ell-spectra}, to prove the asserted homotopy equivalence of spectra it suffices
to prove that 
$$||\cJ_\ell ||, \ ||\cJ_\ell || \circ (\psi^p)_\ell: ||\ul\cB GL(\bC))_\ell || \ \to \ ||\ul\cB G_\ell(S^2)||$$
are homotopy equivalent maps of spectra.  By Theorem 8.2, once we apply  $|| - ||$ to the diagram of (\ref{eqn:F-J})
we may add to the diagram  by inserting  
$(\psi^p)_\ell: ||\ul\cB GL(\bC))_\ell || \to ||\ul\cB GL(\bC))_\ell || $ as the left vertical map
while preserving homotopy commutativity.  In particular, we obtain the
required homotopy commutative triangle
\begin{equation}
\label{eqn:F-J-||}
\xymatrix{
||(\ul\cB GL(\bC))_\ell || \ar[rd]^{||\cJ_\ell ||} \ar[dd]^{(\psi^p)_\ell} & \\ 
& ||\ul\cB G_\ell(S^2)|| \ .\\
(\ul\cB GL(\bC))_\ell \ar[ru]_{||\cJ_\ell ||}  & 
}
\end{equation}
\end{proof}

\vskip .1in

As mentioned in the introduction, we denote by $(\bf{kU})^o$ the 1-connected (and hence, 2-connected) 
cover of $\bf{kU}$ and denote by ${\bf bS_\ell^2}$  the 2-connected cover 
of $\bf BS_\ell^2$.    The following result follows
quickly from Theorem \ref{thm:stable-adams} and Proposition \ref{prop:fingen}.

\begin{cor}
\label{cor:oriented-stable-adams}
Let $(\bf J_\ell)^o:  (({\bf kU})^o)_\ell \ \to \ \bf{bS^2_\ell}$ denote the map induced by ${\bf J_\ell}$ 
 on 2-connnected covers.  The following maps 
  $$(\bf J_\ell)^o, \ (\bf J_\ell)^o \circ (\psi^p)_\ell:  (({\bf kU})^o)_\ell 
\quad \to \quad {\bf bS^2_\ell}  $$
are homotopy equivalent as maps of spectra whenever $p$ and $\ell$ are distinct primes.
\end{cor}

\begin{proof}
Let $[{\bf A},{\bf B}]$ denote the group of homomorphisms from a spectrum ${\bf A}$ to a spectrum 
${\bf B}$.  We have an exact sequence
\begin{equation}
\label{ex:exact-spec}
[({\bf kU})^o)_\ell,K(\bZ_\ell,0)] \ \to \  [({\bf kU})^o)_\ell,\bf{bS_\ell ^2}] \ \to 
[({\bf kU})^o)_\ell,{\bf BS_\ell ^2}]. 
\end{equation}  
(See Proposition \ref{prop:fingen}.) \ Since the two maps
$$
(\bf J_\ell), \ (\bf J_\ell) \circ (\psi^p)_\ell \in [({\bf kU})_\ell\times_{{\bf K(\bZ_\ell,0)}} {\bf K(\bZ,0)},\bf{bS_\ell^2}]
$$
have equal image in $[({\bf kU})_\ell\times_{{\bf K(\bZ_\ell,0)}} {\bf K(\bZ,0)},\bf{BS_\ell^2}]$ 
thanks to Theorem \ref{thm:stable-adams}, \\
$(\bf J_\ell)^o, \ (\bf J_\ell)^o \circ (\psi^p)_\ell \in [(({\bf kU})^o)_\ell,\bf{bS_\ell^2}]$ have equal image
in $[(({\bf kU})^o)_\ell,\bf{BS_\ell^2}]$.
Since $(({\bf kU})^o)_\ell$ is connected, $[(({\bf kU})^o)_\ell,K(\bZ_\ell,0)] = 0$.  Consequently,
the exact sequence tell us 
that 
$$(\bf J_\ell)^o \  = \  (\bf J_\ell)^o \circ (\psi^p)_\ell \quad \in \quad [({\bf kU})^o)_\ell,\bf {bS_\ell^2}].$$
\end{proof}

\vskip .1in

\begin{remark}
\label{rem:false}
The maps  the maps ${\bf J}_\ell, \ {\bf J}_\ell \circ (\psi^p)_\ell$ lift 
to maps
\begin{equation}
\label{eqn:lift-2maps}
\widetilde{\bf J_\ell}, \ \widetilde{\bf J_\ell} \circ (\psi^p)_\ell: ({\bf kU})_\ell \quad \to \quad {\bf bS_\ell^2}.
\end{equation}
Theorem \ref{thm:stable-adams} tells us that the maps of (\ref{eqn:lift-2maps}) become homotopy equivalent
when composed with ${\bf bS_\ell ^2} \to {\bf BS_\ell ^2} $.  Corollary \ref {cor:oriented-stable-adams} tells us
that the maps of (\ref{eqn:lift-2maps}) are homotopy equivalent maps when restricted to $ (({\bf kU})^o)_\ell$.

However, as explained in \cite{B-K}, the maps of (\ref{eqn:lift-2maps}) are not homotopy equivalent (as maps of spectra).
\end{remark}

\vskip .2in


\section{Concluding remark about past arguments}
\label{sec:remarks}

The verification of the Stable Adams Conjecture given in Theorem 10.4 of \cite{F80}
is mistaken in its treatment of orientations.  A computation by an anonymous  mathematician
gave a computation involving Toda brackets which showed that the proof given failed in
the non-oriented context. 

Nevertheless, the overall approach of this paper is similar to that of \cite{F80}.
The careful reader of this paper will find  modifications of the definitions of $X$-fibrations
(using  Reedy sectionings) and of smash product of algebraic spheres.  In this paper,
we introduce and use
$\ell$-completed $X$-fibrations, a class of $X$-fibrations more general than those
obtained as $\ell$-completions of $X$-fibrations.  Further
modifications include changes of notation, organization, exposition, and attention to detail.
 \vskip .2in


\end{document}